\documentclass[final,1p,times]{elsarticle}

\usepackage{amsfonts}
\usepackage{graphicx}
\usepackage{amsmath}
\usepackage{graphicx}
\usepackage{amssymb}
\usepackage[colorlinks]{hyperref}
\usepackage{color}
\usepackage{xcolor}

\newtheorem{theorem}{Theorem}

\newtheorem{case}{Case}

\newtheorem{definition}[theorem]{Definition}

\newtheorem{lemma}[theorem]{Lemma}

\newtheorem{remark}[theorem]{Remark}
\newenvironment{proof}[1][Proof]{\textbf{#1.} }{\ \rule{0.5em}{0.5em}}

\usepackage{amssymb}

\newcommand{\R}{\mathbb{R}}
\newcommand{\T}{\mathbb{T}}
\newcommand{\A}{\mathbb{A}}
\newcommand{\Z}{\mathbb{Z}}
\newcommand{\eps}{\varepsilon}
\newcommand{\cA}{\mathcal{A}}
\newcommand{\correction}[2]{#2}
\newcommand{\cor}[1]{#1}
\newcommand\corTypo[1]{#1}
\newcommand{\mymarginpar}[1]{}

\begin{document}

\begin{frontmatter}



\title{
Computer Assisted Proof of Drift Orbits Along Normally Hyperbolic Manifolds
}


\author[mjc]{Maciej J. Capi\'nski \corref{cor}\fnref{mjcA}}
\address[mjc]{Faculty of Applied Mathematics, AGH University of Science and Technology, al. Mickiewicza 30, 30-059 Krak\'ow, Poland.}
\cortext[cor]{Corresponding author.}
\fntext[mjcA]{Partially supported by the NCN grant 2018/29/B/ST1/00109. The work has been conducted during
the visit to FAU sponsored by the Fulbright Foundation.}

\author[jg]{Jorge Gonzalez\fnref{jgA}} 
\address[jg]{School of Mathematics, Georgia Institute of Technology, 686 Cherry Street, Atlanta, GA, 30332, USA.}
\fntext[jgA]{Partially supported by NSF grant MSPRF DMS-2001758.}

\author[jpm]{Jean-Pierre Marco}
\address[jpm]{Institut de Math\'{e}matiques, Analyse algx\'{e}brique, Universit\'{e} Pierre
et Marie Curie, 175 rue du Chevaleret, 75013 Paris, France.}

\author[jdmj]{J.D. Mireles James\fnref{jdmjA}}
\address[jdmj]{Department of Mathematical Sciences, Florida Atlantic University, 777 Glades Road, Boca Raton, FL 33431, USA. }
\fntext[jdmjA]{Partially supported by NSF grant DMS 1813501.}

\begin{abstract}
Normally hyperbolic invariant manifolds theory provides an efficient tool for proving diffusion in dynamical systems. 
In this paper we develop a methodology for computer assisted proofs of diffusion in a-priori chaotic systems based on 
this approach.
We devise a method, which allows us to validate the needed conditions in a finite number of steps, which can be performed by a computer by means of rigorous-interval-arithmetic computations. We apply our method to the generalized standard map, obtaining diffusion over an explicit range of actions. 
\end{abstract}


\end{frontmatter}



\section{Introduction}
\correction{comment 1}{}

One of the fundamental problems of classical mechanics is to understand 
the dynamics of perturbations of completely integrable systems.
Indeed the solar system itself 
can be viewed as a system of weakly coupled (completely integrable)
two body problems, and the question of \correction{comment 2}{its }stability has captivated mathematicians 
since the days of Newton.
To formalize the discussion
\correction{comment 9}{let $\T = \R \mod 2\pi$}, let $\A^n=T^*\T^n=\T^n\times\R^n$ denote the annulus 
with the angle-action coordinates $(\theta,r)$, 
endowed with the symplectic\footnote{We refer to \cite{LM} for a comprehensive presentation of symplectic geometry} form $\corTypo{\omega}=\sum_{i=1}^ndr_i\wedge d \theta_i$.
Consider a Hamiltonian of the form
\begin{equation}\label{eq:hamform}
H(\theta,r)=h(r)+f(\theta,r),
\end{equation}
where $f$ is small in some suitable function space (analytic, $C^\infty$, $C^\kappa$, etcetera). 
The Hamiltonian equations of motion generated by
$H$ are
\begin{align*}
\dot \theta_i &=\partial_{r_i}H(\theta,r)=\partial_{r_i}h(r)+\partial_{r_i}f(\theta,r)\\
\dot r_i & =-\partial_{\theta_i}H(\theta,r)=-\partial_{\theta_i}f(\theta,r).
\end{align*}
We say that the system is completely integrable when 
$f \equiv 0$, as all the orbits move with constant 
velocity on invariant tori. 

When $f$ is small, the evolution of the action variables $r_i$ are ``slow''.
The fact that this evolution is ``extremely slow'' emerged from averaging methods 
originally developed by  Lagrange and Laplace,  furthered by Poincar\'e and Birkhoff,
culminating in the work of
Littlewood  \cite{Littlewood_1959} and in 
the major achievements of Nekhoroshev  \cite{Nekhoroshev_1977}. 
Thanks to the work of these and many subsequent authors,
it is now well-known that if \correction{comment 3}{$h$} is 
strictly convex and analytic, then the drift in the action variables cannot 
exceed a variation of $O(\eps^{1/2n})$ during an 
$O\big(\exp (1/\eps)^{1/2(n-2)}\big)$-long  time.
Here $\eps$ measures the size of the perturbation function $f$.

Examining the problem from another direction,
Kolmogorov 
\cite{MR0068687} proved the first results on ``perpetual stability'' of solutions of 
analytic systems~\eqref{eq:hamform}. Kolmogorov's approach is geometric in 
essence: he proves that
-- provided \correction{comment 4}{that $f$ is analytic, } that the frequency vector $\nabla h(r^0)$ is 
Diophantine\correction{}{, and that certain non-degeneracy conditions on $h$ hold}
-- then the integrable invariant tori of the form 
$\T^n\times\{r^0\}$
persist and are only slightly deformed when the perturbation $f$ is added 
to the system.  By Diophantine we mean that there are constants 
$\gamma>0$ and $\tau>0$ such that
\[
\vert k\cdot \nabla h(r^0)\vert \geq \frac{\gamma}{\Vert k\Vert^\tau},\qquad \forall k\in\Z^n\setminus\{0\}.
\]
Arnold and Moser then added their own contributions
to this initial result, giving rise to what is now known as the KAM theory
\cite{MR0163025,MR0208078,MR0170705,MR1345153}.  See also
\cite{MR1858536} for much more complete discussion of the KAM theory and 
and its development.

\correction{comment 5}{Taken together, the KAM and averaging theories 
provide indispensable information about perturbations of  
integrable Hamiltonian systems.
The KAM theory says that 
\textit{some} orbits remain close to orbits of the unperturbed system
for all time (the KAM tori), while the averaging theory says that
\textit{all} orbits stay close to the unperturbed level sets of the Hamiltonian
for exponentially long times. 
A natural question is to ask \textit{do there exist orbits whose Hamiltonian
increases (or decreases) by an ``arbitrarily'' large  amount,
on a long enough time scale?}
}

The first explicit example actually exhibiting this phenomenon was given by Arnold in \cite{Ar}, 
and had the form: 
\begin{equation}\label{eq:hamform2}
H_{\eps}(\theta,r)=r_0+\tfrac{1}{2} (r_1^2+r_2^2)+\mu\cos \theta_2+\eps g(\theta,r),\quad 
\theta\in\T^3,\quad r\in\R^3,
\end{equation}
where $g$ is an explicit fixed
trigonometric polynomial, and $\mu$ and $\eps$ are independent parameters.
The example has several important special properties, such as:
\begin{itemize}
\item When $\mu=\eps=0$, the system reduces to $h$ and is completely integrable 
in angle-action form.
\item When  $\mu>0$ and $\eps=0$, the system $H_0$ is Liouville-integrable. 
In particular, it admits a normally hyperbolic (and symplectic) invariant annulus
$\cA_0=\A^2\times\{O\}$, where $O=(0,0)\in\A$ is the hyperbolic fixed point of the pendulum
$\tfrac{1}{2} r_2^2+\mu\cos \theta_2$. The stable and unstable manifolds of $\cA_0$ take the  form
$W^\pm(\cA_0)=\A^2\times W^\pm(O)$. The Hamiltonian flow in restriction to $\cA_0$ is 
completely integrable, in the sense that it admits a foliation by the Lagrangian
(for the induced structure) invariant  tori 
$\big(\T^2\times\{(r_0,r_1)\}\big)_{(r_0,r_1)\in\R^2}$.
\item For fixed $\mu$ and small enough \corTypo{$\varepsilon$ }
($\eps$ has to be exponentially small w.r.t. $\mu$ in Arnold's example), 
 the annulus $\cA_0$ is only slightly deformed and gives rise to a $4$-dimensional 
 normally hyperbolic (symplectic) invariant  annulus $\cA_\eps$ close to $\cA_0$, with 
 a rich homoclinic structure, while the Hamiltonian flow
on $\cA_\eps$ is close to completely 
integrable. 
\end{itemize}

It is important to stress that 
the perturbation $g$ is carefully chosen in Arnold's example, so that the annulus $\cA_0$
is still invariant when $\eps>0$ \correction{comment 6}{and the dynamics on $\cA_0$ remains unchanged. Moreover, the perturbation does not depend on the action variable $r$ and the transversal intersection of stable and unstable manifolds can be proved for all invariant circles on $\cA_0$.} By exploiting these facts 
Arnold was able to show that for $\mu, \varepsilon >0$ small enough,
$H_\eps$  admits a solution 
$\gamma_\eps(t)=\big(\theta(t),r(t)\big)$ which drifts of
order $1$ in action for  
suitable (very large) $T_{\mu,\eps}$.  That is
\[
r_1(0)<0, \qquad r_1(\cor{T_{\mu,\eps}})>1,
\]
for this orbit.  This provided the first explicit example 
where orbits of the perturbed system 
``diffuse'' as far and as fast
\footnote{The fact that the speed of Arnold diffusion coincides with the prediction
of averaging theory was indeed proved much later, see 
\cite{MR1986314,MR2152466,MR2739833,MR3670012}\corTypo{.}}
from the unperturbed level sets as allowed by averaging theory.  

The use of two independent 
parameters 
in Arnold's example
simplifies a lot the study: \eqref{eq:hamform2} is to be compared with \eqref{eq:hamform},
where the size of $f$ is the {\em only} available parameter.  Nevertheless, Arnold's example
became a jumping off point for a large body of work. By now this is a thriving industry and
it is known that diffusion occurs under a wide variety of hypotheses.  

Another (deeper) question raised by Arnold is the case where the \correction{comment 7}{unperturbed }Hamiltonian 
is completely integrable and in action-angle form (the famous ``fundamental problem of dynamics'' 
of Poincar\'e). Given a Hamiltonian system $h$ which depends only on the actions,
does there  exist a large (residual) set of perturbations $g$ such that
orbits diffuse in the previous fashion - or even visit any prescribed collection of open subsets
of an energy level?
It turns out that this question is extremely delicate,
and there are still many important open problems in this active area 
of research.  
The present discussion is by no means intended as a literature 
review of the field, we refer to \cite{MR3646879} for a very nice result in any dimension, together with 
relevant references.



Another line of study
comes from weakening the hypothesis that the unperturbed system is completely integrable.
 Consider for example systems of the form \eqref{eq:hamform2}, in 
which the parameter $\mu$ is fixed but not small (say $\mu=1$). Such systems are  
referred to as {\em a priori} unstable, since they already admit hyperbolic invariant objects when 
$\eps = 0$.  The main difficulty in studying {\em a priori} unstable systems is their ``singular character''
or lack of transversality, 
coming from the fact that the manifolds $W^\pm(\cA_0)$ coincide when $\eps=0$. Detecting 
homoclinic intersections in such systems
for generic $g$ when $\eps\neq 0$ is far from trivial and requires new ingredients from variational
methods, weak KAM theory (both in the convex case) or symplectic topology in the general case.

This complication motivated the introduction of a still less degenerate class of 
examples, for which $W^\pm(\cA_0)$ transversely intersect even in the 
case $\eps=0$. This class of systems is known as 
{\em a priori} chaotic, and is the main topic of the present work
(see \cite{GT} and \cite{Ma1} for examples in this category 
closely related to Arnold's).  
Studying such systems is simpler, which leaves open the possibility of asking new and
more quantitative questions, e.g. what is the threshold in $\eps$ under which diffusion 
phenomenons can appear, or, what is the
maximal length of diffusive trajectories? These questions require new 
methods, and it turns out that in realistic physical systems the relevant quantities to 
estimate are difficult to compute.  Our
aim is to provide an explicit example illustrating the relevance of 
computer-assisted methods of proof in such problems.

To simplify the construction we shift our focus to 
symplectic maps instead of Hamiltonian vector fields. 
This reduction is natural, since taking a Poincar\'{e} section
in an energy manifold results in a symplectic 
diffeomorphism.  The main example of the paper is
the  family of symplectic diffeomorphisms 
$f_{\varepsilon} \colon \R^2\times\T^2 \to \R^2\times\T^2$ defined by
\begin{equation}
f_{\varepsilon}\left(  x,y,\theta,I\right)  =\left(
\begin{array}
[c]{l}%
x+y+\alpha\sin\left(  x\right) \\
y+\alpha\sin\left(  x\right) \\
\theta+I\\
I
\end{array}
\right)  +\varepsilon\left(
\begin{array}
[c]{c}%
\cos(x)\sin(\theta)\\
\cos(x)\sin(\theta)\\
\sin(x)\cos(\theta)\\
\sin(x)\cos(\theta)
\end{array}
\right)\corTypo{,}   \label{eq:example}%
\end{equation}
where $(x,y)\in\R^2$ and $(\theta,I)\in \T^2$. Observe that the map $f_\eps$ 
can be seen as a perturbation of a standard map (variables $(x,y)$) coupled to
an $I$-parametrized rotation on $\T^2$ (variables $(\theta,I)$).  
Indeed, when $\varepsilon =0 $ the two systems do not interact and the dynamics is 
a product.  

\begin{figure}
\begin{center}
\includegraphics[width=7cm]{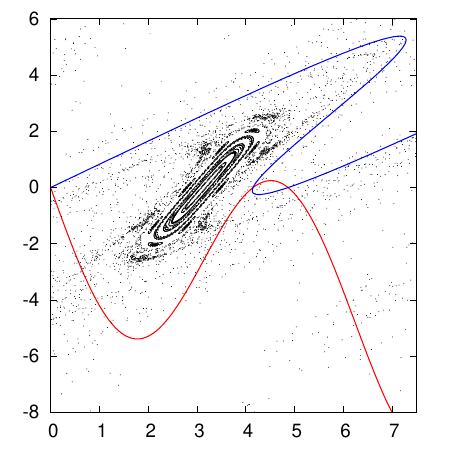}
\end{center}
\caption{Phase space structure of the Chirikov Standard Map when $\alpha=4.$
Black dots indicate the dynamics of a number of ``typical'' orbits. The stable
and unstable manifolds of the fixed point at the origin are depicted by the
red and blue curves respectively.}%
\label{fig:phaseSpace}%
\end{figure}

Of particular interest, the 
standard map has a hyperbolic fixed point
at the origin $O$ in $\R^2$.
In the present work we do not treat $\alpha$ as a perturbation parameter,
and will show that for 
fixed $\alpha\cor{=4}$ the stable and unstable manifolds 
intersect transversely at some point $P$\correction{comment 8}{; see Figure \ref{fig:phaseSpace}} (so that the parameter $\alpha$ plays the role
of $\mu$ in Arnold's example). Consequently, 
$f_0$ admits an invariant
torus $\{O\}\times\T^2$, which is readily seen to be normally hyperbolic, 
and whose stable and unstable manifolds intersect transversely
along a homoclinic torus $\{P\}\times\T^2$. By the Birkhoff-Smale theorem, 
a large enough iterate of the standard map admits a 
horseshoe (homeomorphic to $\{0,1\}^\Z$ endowed with the product topology) near the origin. 
Consequently, for $N$ large enough,  the coupling $f_0^N$ admits a fibered horseshoe, close to $\{0\}\times\T^2$ and homeomorphic to 
$\{0,1\}^\Z\times\T^2$, on which it induces a fiber-preserving dynamics. 
This problem was formalized in \cite{Ma2}.

When $\eps>0$, the preservation of the \corTypo{fibres }is broken, and nothing 
prevents the orbits from drifting along the base $\T^2$ in
the $I$ direction.   
In this paper we use constructive computer assisted arguments to prove that such drift orbits do 
indeed exist for $f_\varepsilon$, and that they have lengths independent of the size $\eps$ of the perturbation.
This makes the system a significant example in the 
\textit{a-priori} chaotic case.  Moreover, the present work provides a self contained exposition of 
 constructive computer assisted methods for proving the existence of diffusion phenomena
in explicit examples.

Our results are based on shadowing theorems for scattering maps worked out in \cite{GLS}. 
A scattering map is a function from a normally hyperbolic invariant manifold to itself, defined 
through appropriate intersections of \corTypo{fibres }of its stable and unstable manifolds. In \cite{GLS} 
it is shown that pseudo orbits resulting from iterations of scattering maps are shadowed 
by true orbits of the system.  We use this method in our main results, which are contained in 
Theorems \ref{th:main}, \ref{cor:strip}, \ref{cor:strip_down} and \ref{th:shadowing-seq}. 
Theorems \ref{th:main}, \ref{cor:strip}, \ref{cor:strip_down} establish orbits which diffuse 
over an explicit interval of actions.  
Theorem \ref{th:shadowing-seq} establishes orbits which shadow sequences of actions, 
chosen from the interval.  The aim of this paper is to provide tools which can be used to obtain 
computer assisted proofs. To check the hypotheses of our theorems one needs to compute the scattering 
maps of the unperturbed system, and to check certain explicit inequalities which measure the 
influence of the perturbation on the action. This influence is computed by considering finite 
fragments of homoclinic orbits. 
We apply our results to give a computer-assisted proof of diffusion for the system 
given by Equation \eqref{eq:example}. 
In a forthcoming paper we plan an application to the Planar Restricted Three Body Problem, 
with mass parameters of the Jupiter-Sun system.

An alternative approach for computer assisted proof of diffusion is given in \cite{CG}. \corTypo{The work just cited }is based on the method of correctly aligned windows. The difference compared to \corTypo{the present work }is that  \cite{CG} requires an explicit construction of `connecting sequences' of windows. 
These windows are then used for shadowing arguments. Here we establish transversal 
intersections of stable/unstable manifolds leading to scattering maps, and check our conditions 
along homoclinic orbits. The shadowing is automatically ensured by \cite{GLS}.

The remainder of the paper is organized as follows.  In Section \ref{sec:prel} we
 review some preliminary 
information about normally hyperbolic invariant manifolds, scattering maps, and 
the interval Newton method.  
In Section \ref{sec:main} we lay out our main theoretical results, namely the constructive 
hypothesis which are used to establish Arnold diffusion in explicit examples.
Section \ref{sec:example} applies the method to the example system.
Proofs of some of the theorems and lemmas are relegated to the Appendices.


\section{Preliminaries\label{sec:prel}}

Throughout the paper, for $x\in \mathbb{R}^{n}$, by $\left\Vert x\right\Vert 
$ we shall mean the Euclidean norm. 
For a set $A$ in a topological space we shall write $%
\overline{A}$ to denote its closure.

\subsection{Normally hyperbolic invariant manifolds\label{sec:nhim}}

In this section we recall the notion of a normally hyperbolic invariant
manifold and state the main result concerning its persistence under small
perturbation. A classic reference for this material is \cite{MR292101,MR287106}.

\begin{definition}
\label{def:nhim} 
\correction{correction 1}{Let $M$ be a smooth $n$-dimensional manifold, and let  $f : M \to M$} be a $C^r$ diffeomorphism, \correction{comment 10}{with $r> 1$}.
Let $\Lambda \subset \cor{M}$ be a compact manifold
without boundary, invariant under $f$, i.e., $f(\Lambda )=\Lambda $. We say that $\Lambda $ is a normally hyperbolic invariant
manifold (with symmetric rates) if there exists a constant $C>0,$ rates $%
0<\lambda <\mu ^{-1}<1$ and a $Tf$ invariant splitting for every $x\in
\Lambda $%
\begin{equation*}
\cor{T_x M}=E_{x}^{u}\oplus E_{x}^{s}\oplus T_{x}\Lambda \corTypo{,}
\end{equation*}%
such that%
\begin{align}
v& \in E_{x}^{u}\Leftrightarrow \left\Vert Df^{k}(x)v\right\Vert \leq
C\lambda ^{-k}\left\Vert v\right\Vert ,\quad k\leq 0,
\label{eq:rate-cond-nhim1} \\
v& \in E_{x}^{s}\Leftrightarrow \left\Vert Df^{k}(x)v\right\Vert \leq
C\lambda ^{k}\left\Vert v\right\Vert ,\quad k\geq 0,
\label{eq:rate-cond-nhim2} \\
v& \in T_{x}\Lambda \Rightarrow \left\Vert Df^{k}(x)v\right\Vert \leq C\mu
^{|k|}\left\Vert v\right\Vert ,\quad k\in \mathbb{Z}.
\label{eq:rate-cond-nhim3}
\end{align}
\end{definition}

Let $d\left( x,\Lambda \right) $ stand for the distance between a point $x$
and the manifold $\Lambda $. 
Given a normally
hyperbolic invariant manifold and a suitable small tubular neighbourhood $%
U\subset \cor{M}$ of $\Lambda $ one defines its local unstable and
local stable manifold \cite{MR292101} as%
\begin{align*}
W_{\Lambda }^{u}\left( f,U\right) & =\left\{ y\in \cor{M}\,|\,f^{k}(y)\in U,d\left( f^{k}(y),\Lambda \right) \leq C_{y}\lambda
^{\left\vert k\right\vert },\,k\leq 0\right\} , \\
W_{\Lambda }^{s}\left( f,U\right) & =\left\{ y\in \cor{M}\,|\,f^{k}(y)\in U,d\left( f^{k}(y),\Lambda \right) \leq C_{y}\lambda
^{k},\,k\geq 0\right\} ,
\end{align*}%
where $C_{y}$ is a positive constant, which can depend on $y$. We define the
(global) unstable and stable manifolds as%
\begin{equation*}
W_{\Lambda }^{u}\left( f\right) =\bigcup_{n\geq 0}f^{n}\left( W_{\Lambda
}^{u}\left( f,U\right) \right) ,\qquad W_{\Lambda }^{s}\left( f\right)
=\bigcup_{n\geq 0}f^{-n}\left( W_{\Lambda }^{s}\left( f,U\right) \right) .
\end{equation*}

The manifolds $W_{\Lambda }^{u}\left( f,U\right) $, $W_{\Lambda }^{s}\left(
f,U\right) $, $W_{\Lambda }^{u}\left( f\right) $ and $W_{\Lambda }^{s}\left(
f\right) $ are foliated by%
\begin{align*}
W_{x}^{u}\left( f,U\right) & =\left\{ y\in\cor{M}\,|\,f^{k}(y)\in
U,d(f^{k}(y),f^{k}(x))\leq C_{x,y}\lambda ^{\left\vert k\right\vert
},\,k\leq 0\right\} , \\
W_{x}^{s}\left( f,U\right) & =\left\{ y\in \cor{M}\,|\,f^{k}(y)\in
U,d(f^{k}(y),f^{k}(x))\leq C_{x,y}\lambda ^{k},\,k\geq 0\right\} ,
\end{align*}%
where $x\in \Lambda $ and $C_{x,y}$ is a positive constant, which can depend
on $x$ and $y$,%
\begin{equation*}
W_{x}^{u}\left( f\right) =\bigcup_{n\geq 0}f^{n}\left( W_{f^{-n}\left(
x\right) }^{u}\left( f,U\right) \right) ,\qquad W_{x}^{s}\left( f\right)
=\bigcup_{n\geq 0}f^{-n}\left( W_{f^{n}\left( x\right) }^{s}\left(
f,U\right) \right) .
\end{equation*}

Let%
\begin{equation}
l<\min \left\{ r,\frac{|\log \lambda |}{\log \mu }\right\} .
\label{eq:nhim-smoothness}
\end{equation}%
The manifold $\Lambda $ is $C^{l}$ smooth, the manifolds $W_{\Lambda
}^{u}\left( f\right) ,W_{\Lambda }^{s}\left( f\right) $ are $C^{l-1}$ and $%
W_{x}^{u}\left( f\right) $, $W_{x}^{s}\left( f\right) $ are $C^{r}$ \cite%
{DLS1}. Normally hyperbolic manifolds, as well as their stable and unstable manifolds and their fibres persist under small perturbations \cite{MR292101}.



\subsection{Shadowing of scattering maps}\label{sec:shadowing}

Our diffusion result is based on shadowing lemmas for scattering maps found
in \cite{GLS}, which we now summarize.

\cor{Let $(M,\omega)$ be a smooth symplectic manifold. }Let us assume that $\Lambda\cor{\subset M}$ is a normally hyperbolic invariant manifold for 
\correction{comment 25}{a $C^r$ symplectic map $f:M\to M$, where $r>1$. We assume that $\Lambda$ is even dimensional and symplectic with the symplectic form $\omega|_{\Lambda}$, then $f|_{\Lambda}$ is symplectic on $\Lambda$. }
We define two maps, 
\begin{align*}
\Omega_{+} & :W_{\Lambda}^{s}(f)\rightarrow\Lambda, \\
\Omega_{-} & :W_{\Lambda}^{u}\left( f\right) \rightarrow\Lambda,
\end{align*}
where $\Omega_{+}(x)=x_{+}$ iff $x\in W_{x_{+}}^{s}\left( f\right) $, and $%
\Omega_{-}(x)=x_{-}$ iff $x\in W_{x_{-}}^{u}\left( f\right) .$ These are
referred to as the wave maps\correction{comment 12}{.}

\begin{definition}
\label{def:homoclinic-channel} We say that a manifold $\Gamma \subset
W_{\Lambda }^{u}\left( f\right) \cap W_{\Lambda }^{s}\left( f\right) $ is a
homoclinic channel for $\Lambda $ if the following conditions hold:

\begin{itemize}
\item[(i)] \label{itm:homoclinic-channel-c1} for every $x\in \Gamma $ \correction{comment 13}{}
\begin{align}
T_{x}W_{\Lambda }^{s}\left( f\right) \cor{+} T_{x}W_{\Lambda }^{u}\left(
f\right) & =\cor{T_x M},  \label{eq:scatter-cond-1} \\
T_{x}W_{\Lambda }^{s}\left( f\right) \cap T_{x}W_{\Lambda }^{u}\left(
f\right) & =T_{x}\Gamma ,  \label{eq:scatter-cond-2}
\end{align}

\item[(ii)] \label{itm:homoclinic-channel-c2} the fibres of $\Lambda$
intersect $\Gamma$ transversally in the following sense%
\begin{align}
T_{x}\Gamma\oplus T_{x}W_{x_{+}}^{s}\left( f\right) & =T_{x}W_{\Lambda
}^{s}\left( f\right) ,  \label{eq:scatter-cond-3} \\
T_{x}\Gamma\oplus T_{x}W_{x_{-}}^{u}\left( f\right) & =T_{x}W_{\Lambda
}^{u}\left( f\right) ,  \label{eq:scatter-cond-4}
\end{align}
for every $x\in\Gamma$,

\item[(iii)] the wave maps $(\Omega_{\pm})_{\mid\Gamma}:\Gamma\rightarrow
\Lambda$ are diffeomorphisms onto their image.
\end{itemize}
\end{definition}

\begin{definition}\label{def:scattering-map} \correction{comment 14}{}
Assume that $\Gamma$ is a homoclinic channel for $\Lambda$ and let 
\begin{equation*}
\Omega_{\pm}^{\Gamma}:=\left( \Omega_{\pm}\right) |_{\Gamma}.
\end{equation*}
We define a scattering map $\sigma^{\Gamma}$ for the homoclinic channel $%
\Gamma$ as 
\begin{equation*}
\sigma^{\Gamma}:=\Omega_{+}^{\Gamma}\circ\left( \Omega_{-}^{\Gamma}\right)
^{-1}:\Omega_{-}^{\Gamma}\left( \Gamma\right) \rightarrow\Omega_{+}^{\Gamma
}\left( \Gamma\right) .
\end{equation*}
\end{definition}

\correction{comment 25}{We have the following symplectic property of the scattering map.
\begin{theorem}\label{th:symplectic} \cite{GLS} Assume that $M$ is endowed with a symplectic form $\omega$ and that $\omega |_{\Lambda}$ is also symplectic.
Assume that $f$ is symplectic.
Assume that there exists a homoclinic channel $\Gamma$ and so the scattering map $\sigma^{\Gamma}$ is well defined. Then, the scattering map $\sigma^{\gamma}$ is symplectic.
\end{theorem}}
We have the following theorem\corTypo{, which is the main tool which we use to obtain our results}.

\begin{theorem}
\cite{GLS}\label{th:shadowing} Assume that $f:\cor{M}\rightarrow 
\cor{M}$ is a sufficiently smooth map, $\Lambda\subset\cor{M}$
is a normally hyperbolic invariant manifold with stable and unstable
manifolds which intersect transversally along a homoclinic channel $\Gamma
\subset\cor{M},$ and $\sigma$ is the scattering map associated to $%
\Gamma$.


Let $m_{1},\ldots ,m_{\cor{l}}\in \mathbb{N}$ be a fixed sequence of integers. Let 
$\left\{ x_{i}\right\} _{i=0,\ldots ,\cor{l}}$ be a finite pseudo-orbit in $%
\Lambda $, that is a sequence of points in $\Lambda $ of the form \correction{comment 15}{}
\begin{equation}
x_{i+1}=f^{m_{i}}\circ \sigma ^{\Gamma }\left( x_{i}\right) ,\qquad
i=0,\ldots ,\cor{l}-1,\,\cor{l}\geq 1.  \label{eq:pseudo-orbit}
\end{equation}%
Then for every $\delta>0$ there exists an orbit $\left\{ z_{i}\right\}
_{i=0,\ldots,\cor{l}}$ of $f$ in $\cor{M}$, with $z_{i+1}=f^{k_{i}}\left(
z_{i}\right) $ for some $k_{i}>0$, such that $d\left( z_{i},x_{i}\right)
<\delta$ for all $i=0,\ldots,\cor{l}$.
\end{theorem}

\begin{remark}\label{rem:th-simplified} \correction{comment 26}{
The original statement of Theorem \ref{th:shadowing} from \cite{GLS} does not require compactness of $\Lambda$.
The assumptions are that $f$ preserves measure absolutely continuous with respect to the Lebesgue measure on $\Lambda $, and that $\sigma $ sends positive measure sets to positive measure sets.
Moreover, it is assumed that the pseudo-orbit is contained in some open set $\mathcal{U}\subset \Lambda $ with almost every point of $\mathcal{U}$ recurrent for $f|_{\Lambda }$. In our case, since 
$\Lambda$ is compact and $ f|_{\Lambda}$ is symplectic, almost every point in $\Lambda$ is recurrent for $f|_{\Lambda}$, so we can take $\mathcal{U}=\Lambda$ and simplify the statement of the theorem.}
\end{remark}

\begin{remark}
In \cite{GLS} the statement of the theorem is for pseudo-orbits of the form $%
x_{i+1}=\sigma ^{\Gamma }\left( x_{i}\right) $. Here we shadow pseudo-orbits
of the form (\ref{eq:pseudo-orbit}), but this is the same result as that
from \cite{GLS} for the following reason.

The proof of the theorem in \cite{GLS} is based on a general shadowing lemma 
\cite[Lemma 3.1]{GLS} which ensures that given a pseudo-orbits of the form $%
y_{i+1}=f^{k_{i}}\circ \sigma ^{\Gamma }\circ f^{n_{i}}(y_{i})$ where the
numbers of iterates $k_{i}$, $n_{i}$ are big enough, we are able to find an
orbit of the form $z_{i+1}=f^{k_{i}+n_{i}}(z_{i})$, $\delta $-close to the
pseudo-orbit $y_{i}$.

The shadowing of a pseudo-orbit $x_{i+1}=\sigma ^{\Gamma }\left(x_{i}\right) $ is proven in \cite{GLS} by combining \cite[Lemma 3.1]{GLS} with recurrence. First, by using recurrence, a pseudo-orbit of the form $y_{i+1}=f^{k_{i}}\circ \sigma ^{\Gamma }\circ f^{n_{i}}(y_{i})$ is constructed close to the pseudo-orbit $x_{i+1}=\sigma ^{\Gamma }\left(x_{i}\right) $. The $k_{i}$, $n_{i}$ are chosen to be big enough to apply from \cite[Lemma 3.11]{GLS}. The true orbit, which follows from \cite[Lemma 3.11]{GLS}, shadows the pseudo orbit $y_i$, but since this lies close to $x_i$ one obtains the shadowing of the pseudo orbit $x_{i+1}=\sigma ^{\Gamma }\left(x_{i}\right)$.

The proof of the shadowing of a pseudo-orbit of the form (\ref{eq:pseudo-orbit}) follows from the same construction: One can use
recurrence to construct a pseudo-orbit of the form $y_{i+1}=f^{k_{i}}\circ
\sigma ^{\Gamma }\circ f^{n_{i}}(y_{i})$, so that $y_i$ are close to $x_i$ from (\ref{eq:pseudo-orbit}). The lemma \cite[Lemma 3.1]{GLS} ensures that $y_i$ can be shadowed by a true orbit. Since $y_i$ is close to the pseudo-orbit $x_i$ form (\ref{eq:pseudo-orbit}) we obtain the shadowing of (\ref{eq:pseudo-orbit}) by a true orbit.
\end{remark}

\begin{remark}\label{rem:finite-number-scatter}
The result can be immediately extended to the case where we have a finite
number of scattering maps $\sigma _{1},\ldots ,\sigma _{L}$ to shadow\mymarginpar{\hyperlink{Return_comment 15}{comment 15}}
\begin{equation*}
x_{i+1}=f^{m_{i}}\circ \sigma _{\alpha _{i}}\left( x_{i}\right) ,\qquad
i=0,\ldots ,\cor{l}-1,\,\cor{l}\geq 1,
\end{equation*}%
for two prescribed sequences \correction{comment 16}{$m_{1},\ldots ,m_{\cor{l}}\in \mathbb{N}$ }and $\alpha
_{1},\ldots ,\alpha _{\cor{l}}\in \left\{ 1,\ldots ,L\right\} $; see \cite[Theorem
3.7]{GLS}.
\end{remark}



\section{Main results\label{sec:main}}

\label{sec:main-results}

Let $f_{0},g:\mathbb{R}^{2d}\times\mathbb{T}^{2}\rightarrow\mathbb{R}%
^{2d}\times\mathbb{T}^{2}$ and consider the following system 
\begin{equation*}
f_{\varepsilon}\left( u,s,I,\theta\right) =f_{0}(u,s,I,\theta)+\varepsilon
g\left( u,s,I,\theta\right) ,
\end{equation*}
where $u,s \in \mathbb{R}^d$, $\theta,I \in \mathbb{T}$. Assume that $%
f_{\varepsilon}$ are \correction{comment 20}{smooth }symplectic maps for a symplectic form \correction{comment 28}{$\omega= du \wedge ds + dI \wedge d\theta$, }assume
that for $\varepsilon=0$ 
\begin{equation*}
\Lambda_{0}=\left\{ \left( 0,0,I,\theta\right) :I,\theta\in\mathbb{T}%
^{1}\right\} \simeq \mathbb{T}^2\corTypo{,}
\end{equation*}
is a normally hyperbolic invariant manifold,
and that $I$ is a constant of motion for the unperturbed
system, i.e. 
\begin{equation}
\pi_{I}f_{0}\left( x\right) =\pi_{I}x,  \label{eq:energy-preserved}
\end{equation}
for any $x\in\mathbb{R}^{2d}\times\mathbb{T}^{2}$, where $\pi_I(u,s,I,\theta)=I$. 

\correction{comment 20}{}
\cor{
\begin{remark} \label{rem:why-smooth}We assume smoothness of the maps since our main tool for the proof will be Theorem \ref{th:shadowing}, which requires sufficient smoothness.
\end{remark}}

\begin{remark}
\label{rem:cylinder} The assumption that $\Lambda _{0}$ is a torus
simplifies the arguments, as $\Lambda _{0}$ is compact without boundary and
the normally hyperbolic manifold theorem ensures that $\Lambda _{0}$ is
perturbed to a nearby compact normally \cor{hyperbolic }invariant manifold $\Lambda
_{\varepsilon }$. Having \cor{compactness }of $\Lambda _{\varepsilon }$ is
convenient, but not necessary \correction{comment 18}{if one applies non-compact versions of the normally hyperbolic theorem \cite{MR1445489,MR1686965,MR2439610,MR3098498}}. 
\end{remark}

Our objective is to provide conditions under which for any sufficiently
small $\varepsilon>0$ there exists a point $x_{\varepsilon}$ and a number of
iterates $n_{\varepsilon}$ for which 
\begin{equation}
\pi_{I}\left( f_{\varepsilon}^{n_{\varepsilon}}\left( x_{\varepsilon
}\right) -x_{\varepsilon}\right) >1.  \label{eq:diffusion-orbit-condition}
\end{equation}

The coordinates have the following roles. The $u,s$ are the coordinates on
unstable and stable bundles, respectively, of $\Lambda _{0}$. The $\theta \ $%
is an angle and $I$ plays the role of \corTypo{a constant }of motion for $\varepsilon
=0$. In the setting of action-angle coordinates, the $I$ would be chosen as
the action. We shall refer to $I$ as an `action', slightly abusing the
terminology. In this paper we restrict to the case where the angle and
action are one dimensional. We do so \corTypo{for the sake of }simplicity\footnote{%
We believe that our methods can be generalised to the higher dimensional
case. We make comments how to do so in Remarks \ref{rem:generalisation1}, %
\ref{rem:generalisation2} after the statements of our results. 
}.

A typical setting where our result \corTypo{applies }is that of a time
dependent perturbation of a Hamiltonian system of the form\correction{comment 29}{}
\begin{equation}
\cor{x^{\prime}=J\nabla_x H_{\varepsilon}(x,t)= J\nabla_x\left( H\left( x\right) +\varepsilon G(x,t)\right) ,}
\label{eq:Hamiltonian-form-of-problem}
\end{equation}
where $H:\mathbb{R}^{2d+2}\rightarrow\mathbb{\cor{R}}$, $G:\mathbb{R}
^{2d+2}\times\mathbb{T}^{1}\rightarrow\mathbb{\cor{R}}$ and\correction{comment 19}{} 
\begin{equation*}
J=\left( 
\begin{array}{ll}
0 & Id \\ 
-Id & 0%
\end{array}
\right) ,\qquad\text{for\qquad}Id=\left( 
\begin{array}{ll}
1 & 0 \\ 
0 & 1%
\end{array}
\right) .
\end{equation*}
In such case we can take $f_{\varepsilon}\left( x\right) =\Phi_{2\pi
}^{\varepsilon}\left( x,t_{0}\right) $, for some $t_{0}\in\lbrack0,2\pi)$,
where $\Phi_{t}^{\varepsilon}\left( x_{0},t_{0}\right) $ stands for the time 
$t$ flow induced by (\ref{eq:Hamiltonian-form-of-problem}) with the initial
condition $\left( x_{0},t_{0}\right) $.

Another possibility is to consider the flow induced by (\ref%
{eq:Hamiltonian-form-of-problem}) in the extended phase space and consider a
section of the form $\Sigma \times \mathbb{T}^1$ in $\mathbb{R}^{2d+2}\times 
\mathbb{T}^1$. Then $f_{\varepsilon}$ can then be chosen as the
section-to-section map along the flow in the extended phases space. The time
coordinate plays the role of the angle $\theta$ and we choose $I$ as the
Hamiltonian \correction{comment 19}{$H$ }of the unperturbed system.


\begin{remark}
\label{rem:cylinder}
A typical setting is when $\Lambda _{0}$ is a normally hyperbolic
invariant cylinder (possibly with a boundary) with $\theta \in \mathbb{T}^{1}
$ and $I\in \mathbb{R}$. We would then have $f_{\varepsilon }:\mathbb{R}%
^{2d}\times \mathbb{R}\times \mathbb{T}\rightarrow \mathbb{R}^{2d}\times 
\mathbb{R}\times \mathbb{T}$. In such case we can consider $I\in \left[ 0,1\right] $
and artificially `glue' the system so that $I$ is in $\mathbb{T}^{1}$ to
apply our result directly\cor{, provided that certain conditions are met. }\mymarginpar{\hyperlink{Return_comment 18}{comment 18}}
Details of how this can be done are found in \ref%
{sec:gluing}. 
\end{remark}

The next theorem is our first main result. It provides conditions for the
existence of orbits which diffuse in $I$.

\begin{figure}[ptb]
\begin{center}
\includegraphics[height=2.8cm]{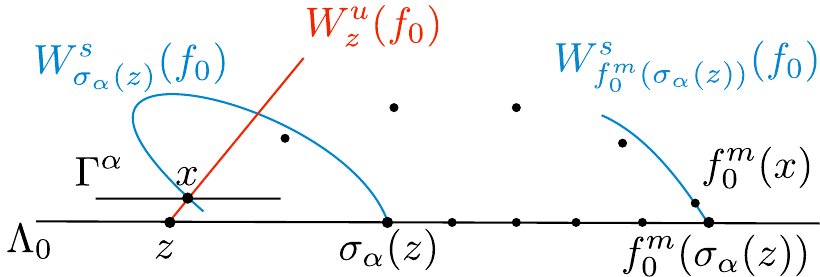}
\end{center}
\caption{The setting for Theorem \protect\ref{th:main}. }
\label{fig:main-thm}
\end{figure}

\begin{theorem}
\label{th:main} 
\cor{Assume that there is a neighborhood $U$ of
$\Lambda_{0}$ and a positive constant $L_{g}$ such that for every $z\in\Lambda_{0}$ and every $x_{u}\in W_{z}^{u}(f_{0},U),$ $x_{s}\in W_{z}^{s}(f_{0},U),$ 
\begin{equation}
\begin{array}
[c]{c}%
\left\vert \pi_{I}\left(  g(0,x_{u})-g\left(  0,z\right)  \right)
\right\vert \leq L_{g}\left\Vert x_{u}-z\right\Vert ,\medskip\\
\left\vert \pi_{I}\left(  g(0,x_{s})-g\left(  0,z\right)  \right)
\right\vert \leq L_{g}\left\Vert x_{s}-z\right\Vert.
\end{array}
\label{eq:Lip-g-assumption}%
\end{equation}
Assume also that there exist positive constants $C,\lambda$, such that $\lambda
\in(0,1)$ and for every $z\in\Lambda_{0}$ and every $x_{u}\in W_{z}^{u}(f_{0}%
,U),x_{s}\in W_{z}^{s}(f_{0},U)$ we have
\begin{equation}%
\begin{array}
[c]{lll}%
\left\Vert f_{0}^{n}\left(  z\right)  -f_{0}^{n}\left(  x_{u}\right)
\right\Vert <C\lambda^{\left\vert n\right\vert } & \qquad & \text{for all
}n\leq0,\medskip\\
\left\Vert f_{0}^{n}\left(  z\right)  -f_{0}^{n}\left(  x_{s}\right)
\right\Vert <C\lambda^{n} &  & \text{for all }n\geq0.
\end{array}
\label{eq:contraction-expansion-bounds}%
\end{equation}}
Assume that for $\varepsilon=0$ we have a sequence $\Gamma^{1},\ldots
,\Gamma^{L}\subset U$ of homoclinic channels for $f_{0}$, with corresponding
wave maps $\Omega_{\pm}^{\alpha}:\Gamma^{\alpha}\rightarrow\Lambda_{0}$ and
scattering maps $\sigma_{\alpha}:\mathrm{dom}\left( \sigma_{\alpha}\right)
\rightarrow\Lambda_{0}$ for $\alpha=1,\ldots,L$.

Assume that for every $z\in \Lambda_{0}$
\begin{enumerate} 
\item There exists an $\alpha\in\left\{
1,\ldots,L\right\} $ such that $z\in\mathrm{dom}\left(
\sigma_{\alpha}\right) $.
\item There exists an $m\in\mathbb{N}$ and
a point $x\in\Gamma_{\alpha}$, $x\in W_{z}^{u}\left( f_{0},U\right) \cap
W_{\sigma_{\alpha}\left( z\right) }^{s}\left( f_{0}\right) $ such that $%
f_{0}^{m}\left( x\right) \in W_{f_{0}^{m}\left( \sigma_{\alpha}\left(
z\right) \right) }^{s}(f_{0},U)$ (see Figure \ref{fig:main-thm}) 
and
\begin{equation}
\sum_{j=0}^{m-1}\pi_{I}g\left( f_{0}^{j}\left( x\right) \right) -\frac{%
1+\lambda}{1-\lambda}L_{g}C>0.  \label{eq:key-assumption}
\end{equation}
\end{enumerate}
(The above $\alpha,m$ and $x$ can depend on the choice of $z$.)

Then for sufficiently small $\varepsilon>0$ there exists an $x_{\varepsilon}$
and $n_{\varepsilon}>0$ such that%
\begin{equation*}
\pi_{I}\left( f_{\varepsilon}^{n_{\varepsilon}}\left( x_{\varepsilon
}\right) -x_{\varepsilon}\right) >1.
\end{equation*}
\end{theorem}

Before giving the proof let us make a couple of comments about the
assumptions.

\begin{remark}\label{rem:lip-remarks}\correction{comment 21}{}
\cor{Assumption (\ref{eq:Lip-g-assumption}) will readily hold since $\Lambda_{0}$ is compact, so we can take $\overline{U}$ to be
compact as well and existence of $L_g$ follows from the fact that $g$ is smooth}. Conditions (\ref{eq:contraction-expansion-bounds}) will
hold due to the contraction and expansion properties along the stable and
unstable manifolds. What is important for us is to have explicit bounds $%
L_{g}$,$C$ and $\lambda$ which enter into the key assumption (\ref%
{eq:key-assumption}). 
\end{remark}

\begin{remark} \label{rem:comparison-Melnikov}\correction{comment 30}{}\cor{
Condition (\ref{eq:key-assumption}) measures the
influence of the perturbation term $g$ on the coordinate $I$. This can be
thought of as a discrete analogue of the following formula for the perturbed scattering map from \cite{GLS}
\[\sigma_{\varepsilon} = \sigma_0 + \varepsilon J \nabla S\circ \sigma_0 + O(\varepsilon ^2),\]
where
\begin{align*}S(x)=&\lim_{T\to +\infty}\int_{-T}^0 \left [ \frac{dH_{\varepsilon}}{d\varepsilon}_{|\varepsilon=0}\circ \Phi_t \circ (\Omega_{-}^{\Gamma_0})^{-1}\circ \sigma_0^{-1}(x)
				-\frac{dH_{\varepsilon}}{d\varepsilon}_{|\varepsilon=0}\circ \Phi_t \circ \sigma_0^{-1}(x) \right ] dt \\
				&+\lim_{T\to +\infty}\int^{T}_0 \left [ \frac{dH_{\varepsilon}}{d\varepsilon}_{|\varepsilon=0}\circ \Phi_t \circ (\Omega_{+}^{\Gamma_0})^{-1}(x)
				-\frac{dH_{\varepsilon}}{d\varepsilon}_{|\varepsilon=0}\circ \Phi_t (x) \right ] dt,
\end{align*} and where $\Phi_t$ is the flow of the unperturbed system. 
Instead of computing an integral along a homoclinic orbit of the flow, we compute a sum along a discrete orbit. An important feature is that we are computing it along a finite fragment
of the homoclinic. The second term in (\ref{eq:key-assumption}) takes into
account the truncated tail. We believe that for computer assisted proofs computing a bound for a sum for a discrete dynamical system, which comes from considering a time shift map or Poincar\'e map, is more convenient than computing a bound on an integral over a trajectory of a flow. }
\end{remark}

\begin{remark}
In Theorem \ref{th:main} we assume that the homoclinic channels are in 
$U$, meaning that they are close to $\Lambda_{0}$. This is not a restrictive
assumption, since a homoclinic channel which is far away can be propagated
close to $\Lambda_{0}$ by using backward iterates of $f_{0}$. 
\end{remark}

\begin{remark}
\label{rem:generalisation1} Theorem \ref{th:main} can be generalised
to the setting of higher dimensional $\theta$ and $I$ as follows. If we have
actions $I_1,\ldots,I_k$, we can single out one of them (say $I=I_1$) for
the conditions (\ref{eq:Lip-g-assumption}) and (\ref{eq:key-assumption}), to
obtain diffusion towards the singled out action. 
\end{remark}

\begin{remark} \label{rem:g-term-form}\correction{comment 31}{}We have assumed that $f_{\varepsilon}(x)=f_0(x)+\varepsilon g(x)$. We can assume just as well that $f_{\varepsilon}(x)=f_0(x)+\varepsilon g(\varepsilon,x)$, with smooth $g(\varepsilon,x)$. Then in conditions (\ref{eq:Lip-g-assumption}) and (\ref{eq:key-assumption}) we can write $g(0,\cdot)$ instead of $g(\cdot)$, and the result will follow from the same arguments. Analogous modifications can be made also in subsequent theorems. We consider $g(x)$ instead of $g(\varepsilon,x)$ since it simplifies and shortens the notation.
\end{remark}

\begin{proof}[Proof of Theorem \protect\ref{th:main}]
The manifold $\Lambda_{0}$ is perturbed to a normally hyperbolic invariant
manifold $\Lambda_{\varepsilon}$ for $f_{\varepsilon}$. Moreover, for
sufficiently small $\varepsilon$\corTypo{, }if $z\in\Lambda_{\varepsilon}$, $x_{u}\in
W_{z}^{u}(f_{\varepsilon},U)$ and $x_{s}\in W_{z}^{s}(f_{\varepsilon},U)$\corTypo{, then}
\begin{equation}
\begin{array}{lll}
\left\Vert f_{\varepsilon}^{n}\left( z\right) -f_{\varepsilon}^{n}\left(
x_{u}\right) \right\Vert <C\lambda_{\varepsilon}^{\left\vert n\right\vert }
& \qquad & \text{for all }n\leq0,\medskip \\ 
\left\Vert f_{\varepsilon}^{n}\left( z\right) -f_{\varepsilon}^{n}\left(
x_{s}\right) \right\Vert <C\lambda_{\varepsilon}^{n} &  & \text{for all }%
n\geq0,%
\end{array}
\label{eq:contraction-expansion-bounds-eps}
\end{equation}
with $\lambda_{\varepsilon}$ converging to $\lambda$ as $\varepsilon$ tends
to zero.

Since transversal intersections persist under perturbation, the homoclinic
channels $\Gamma^{1},\ldots,\Gamma^{L}$ for $f_{0}$ are perturbed to
homoclinic channels $\Gamma_{\varepsilon}^{1},\ldots,\Gamma_{%
\varepsilon}^{l} $ for $f_{\varepsilon}$, provided that $\varepsilon>0$ is
sufficiently small. 
This leads \cite{DLS1} to a scattering map \correction{comment 22}{$\sigma_{\alpha
}^{\varepsilon}:\Omega_{-}^{\Gamma_{\varepsilon}^{\alpha}}\left(
\Gamma_{\varepsilon}^{\alpha}\right) \rightarrow\Omega_{+}^{\Gamma
_{\varepsilon}^{\alpha}}\left( \Gamma_{\varepsilon}^{\alpha}\right) $} for $%
f_{\varepsilon}$.

Our first objective is to show that for any $z_{\varepsilon}\in
\Lambda_{\varepsilon}$ there exists an $m\in\mathbb{N}$ and $%
\alpha\in\left\{ 0,\ldots,L\right\} $ (both $m$ and $\alpha$ can depend on $%
z_{\varepsilon}$) such that%
\begin{equation}
\pi_{I}\left( f_{\varepsilon}^{m}\circ\sigma_{\alpha}^{\varepsilon
}(z_{\varepsilon})-z_{\varepsilon}\right) >\varepsilon c,
\label{eq:shadowing-proof-key-step}
\end{equation}
where $c>0$ is a constant, small enough so that we have
\begin{equation}
\sum_{j=0}^{m-1}\pi_{I}g\left( f_{0}^{j}\left( x\right) \right) -\frac{%
1+\lambda}{1-\lambda}L_{g}C>c \corTypo{,} \label{eq:key-assumption-in-proof}
\end{equation}
for any $z\in\Lambda_{0}$ (with the same $c$). We can
find such small $c$ because of (\ref{eq:key-assumption})  and compactness of $\Lambda_{0}$.

\correction{comment 32}{It turns out that (\ref{eq:shadowing-proof-key-step}) is the main step in
our proof, since once it is established the result follows from the
shadowing Theorem \ref{th:shadowing}. Below we first prove (\ref%
{eq:shadowing-proof-key-step}) and then discuss how to apply the shadowing
method.}

Consider now a $z_{\varepsilon}\in\Lambda_{\varepsilon}.$ 
By our assumptions, \correction{comment 33}{for every $z\in\Lambda_0$ }we have an $\alpha\in\left\{
1,\ldots,L\right\} $, $m\in\mathbb{N}$ and $x\in W_{z}^{u}\left(
f_{0},U\right) \cap W_{\sigma_{\alpha}\left( z\right) }^{s}\left(
f_{0}\right) $ such that $f_{0}^{m}\left( x\right) \in
W_{f_{0}^{m}(\sigma_{\alpha}(z)\mathbf{)}}^{s}\left( f_{0},U\right) $ and (%
\ref{eq:key-assumption-in-proof}) holds. This means that for sufficiently small $%
\varepsilon$ 
we shall have \corTypo{an }$x_{\varepsilon}\in
W_{z_{\varepsilon}}^{u}\left( f_{\varepsilon},U\right) $ \corTypo{such that }$%
f_{\varepsilon}^{m}\left( x_{\varepsilon }\right) \in
W_{f_{\varepsilon}^{m}(\sigma_{\alpha}^{\varepsilon}\left(
z_{\varepsilon}\right) )}^{s}\left( f_{\varepsilon},U\right) $, and by (%
\ref{eq:contraction-expansion-bounds-eps}) 
\begin{align}
\left\Vert f_{\varepsilon}^{j}\left( z_{\varepsilon}\right) -f_{\varepsilon
}^{j}\left( x_{\varepsilon}\right) \right\Vert &
<C\lambda_{\varepsilon}^{\left\vert j\right\vert }\qquad\text{for }j\leq0
\label{eq:term-3-helpful-bound} \\
\left\Vert f_{\varepsilon}^{m+j}\left( \sigma_{\alpha}^{\varepsilon}\left(
z_{\varepsilon}\right) \right) -f_{\varepsilon}^{m+j}\left( x_{\varepsilon
}\right) \right\Vert & <C\lambda^{j}_{\varepsilon}\qquad\text{for }j\geq0.
\label{eq:term-1-helpful-bound}
\end{align}
Due to (\ref{eq:key-assumption-in-proof}) and the continuous dependence of $%
x_{\varepsilon},$ $\lambda_{\varepsilon}$ \corTypo{on $\eps$, }for sufficiently small $%
\varepsilon$ we shall have%
\begin{equation*}
\sum_{j=0}^{m-1}\pi_{I}g\left( f_{\varepsilon}^{j}\left( x_{\varepsilon
}\right) \right) -\frac{1+\lambda_{\varepsilon}}{1-\lambda_{\varepsilon}}%
L_{g}C>c.
\end{equation*}

In order to show (\ref{eq:shadowing-proof-key-step}) we will split our
estimates into three terms%
\begin{equation}
f_{\varepsilon}^{m}\left( \sigma_{\alpha}^{\varepsilon}\left( z_{\varepsilon
}\right) \right) -z_{\varepsilon}=\left[ f_{\varepsilon}^{m}\left(
\sigma_{\alpha}^{\varepsilon}\left( z_{\varepsilon}\right) \right)
-f_{\varepsilon}^{m}\left( x_{\varepsilon}\right) \right] +\left[
f_{\varepsilon}^{m}\left( x_{\varepsilon}\right) -x_{\varepsilon}\right] +%
\left[ x_{\varepsilon}-z_{\varepsilon}\right] ,  \label{eq:the-three-terms}
\end{equation}
and investigate bounds on the projection $\pi_{I}$ for each of them. We
start by showing that 
\begin{equation}
\left\vert \pi_{I}\left[ f_{\varepsilon}^{m}\left( \sigma_{\alpha
}^{\varepsilon}\left( z_{\varepsilon}\right) \right)
-f_{\varepsilon}^{m}\left( x_{\varepsilon}\right) \right] \right\vert
\leq\varepsilon \frac{1}{1-\lambda_{\varepsilon}}L_{g}C.
\label{eq:term1-bound}
\end{equation}

Indeed, since $f_{\varepsilon}\left( x\right) =f_{0}\left( x\right)
+\varepsilon g\left( \corTypo{x}\right) $ and $\pi_{I}f_{0}\left( x\right)
=\pi_{I}x$, for any $x_{1},x_{2}$ we have%
\begin{align*}
\pi_{I}f_{\varepsilon}(x_{1})-\pi_{I}f_{\varepsilon}\left( x_{2}\right) &
=\pi_{I}f_{0}(x_{1})+\varepsilon\pi_{I}g\left( x_{1}\right)
-\pi_{I}f_{0}\left( x_{2}\right) -\varepsilon\pi_{I}g\left( x_{2}\right) \\
& =\pi_{I}\left( x_{1}-x_{2}\right) +\varepsilon\pi_{I}\left( g\left(
x_{1}\right) -g\left( x_{2}\right) \right) .
\end{align*}
It follows by induction that%
\begin{equation}
\pi_{I}\left( f_{\varepsilon}^{j}(x_{1})-f_{\varepsilon}^{j}\left(
x_{2}\right) \right) =\pi_{I}\left[ x_{1}-x_{2}\right] +\varepsilon
\sum_{i=0}^{j-1}\pi_{I}\left( g\left( f_{\varepsilon}^{i}(x_{1})\right)
-g\left( f_{\varepsilon}^{i}(x_{2})\right) \right) .
\label{eq:I-difference-by-sum-of-g}
\end{equation}
Taking $x_{1}=f_{\varepsilon}^{m}\left( \sigma_{\alpha}^{\varepsilon}\left(
z_{\varepsilon}\right) \right) $ and $x_{2}=f_{\varepsilon}^{m}\left(
x_{\varepsilon}\right) $ from (\ref{eq:I-difference-by-sum-of-g}) with $%
\pi_{I}\left[ x_{1}-x_{2}\right] $ moved to the left hand side, we have 
\begin{align*}
& \left\vert \pi_{I}\left[ f_{\varepsilon}^{m}\left( \sigma_{\alpha
}^{\varepsilon}\left( z_{\varepsilon}\right) \right)
-f_{\varepsilon}^{m}\left( x_{\varepsilon}\right) \right] \right\vert \\
& =\left\vert \pi_{I}\left( f_{\varepsilon}^{m+j}\left( \sigma_{\alpha
}^{\varepsilon}\left( z_{\varepsilon}\right) \right) )-f_{\varepsilon
}^{m+j}\left( x_{\varepsilon}\right) \right)
-\varepsilon\sum_{i=0}^{j-1}\pi_{I}\left( g\left(
f_{\varepsilon}^{m+i}\left( \sigma_{\alpha }^{\varepsilon}\left(
z_{\varepsilon}\right) \right) \right) -g\left(
f_{\varepsilon}^{m+i}(x_{\varepsilon})\right) \right) \right\vert \\
& <C\lambda_{\varepsilon}^{j}+\varepsilon L_{g}\sum_{i=0}^{j-1}\left\Vert
f_{\varepsilon}^{m+i}\left( \sigma_{\alpha}^{\varepsilon}\left(
z_{\varepsilon}\right) \right) -f_{\varepsilon}^{m+i}(x_{\varepsilon
})\right\Vert \\
& <C\lambda_{\varepsilon}^{j}+\varepsilon
L_{g}\sum_{i=0}^{j-1}C\lambda^{i}_{\varepsilon},
\end{align*}
where the last two inequalities follow from (\ref{eq:term-1-helpful-bound}).
Letting $j\rightarrow\infty$, we obtain (\ref{eq:term1-bound}).

Now consider the third term from (\ref{eq:the-three-terms}). An analogous
bound to (\ref{eq:term1-bound}) is obtained as follows. From (\ref%
{eq:I-difference-by-sum-of-g}) we have that 
\begin{align}
\pi_{I}\left( x_{1}-x_{2}\right) & =\pi_{I}\left[ f_{\varepsilon}^{-j}\left(
x_{1}\right) -f_{\varepsilon}^{-j}\left( x_{2}\right) \right]
+\varepsilon\sum_{i=0}^{j-1}\pi_{I}\left( g\left(
f_{\varepsilon}^{i-j}(x_{1})\right) -g\left(
f_{\varepsilon}^{i-j}(x_{2})\right) \right)  \notag \\
& =\pi_{I}\left[ f_{\varepsilon}^{-j}\left( x_{1}\right) -f_{\varepsilon
}^{-j}\left( x_{2}\right) \right] +\varepsilon\sum_{i=-j}^{-1}\pi _{I}\left(
g\left( f_{\varepsilon}^{i}(x_{1})\right) -g\left(
f_{\varepsilon}^{i}(x_{2})\right) \right).
\label{eq:I-difference-by-sum-of-g-2}
\end{align}
Taking $x_{1}=x_{\varepsilon}$ and $x_{2}=z_{\varepsilon}$, from (\ref%
{eq:I-difference-by-sum-of-g-2}) we obtain%
\begin{align*}
 \left\vert \pi_{I}\left( x_{\varepsilon}-z_{\varepsilon}\right) \right\vert
& \leq\left\vert \pi_{I}\left[ f_{\varepsilon}^{-j}\left( x_{\varepsilon
}\right) -f_{\varepsilon}^{-j}\left( z_{\varepsilon}\right) \right]
\right\vert +\varepsilon\sum_{i=-j}^{-1}\left\vert \pi_{I}\left( g\left(
f_{\varepsilon}^{i}(x_{\varepsilon})\right) -g\left(
f_{\varepsilon}^{i}(z_{\varepsilon})\right) \right) \right\vert \\
& <C\lambda_{\varepsilon}^{j}+\varepsilon L_{g}\sum_{i=-j}^{-1}\left\Vert
f_{\varepsilon}^{i}(x_{1})-f_{\varepsilon}^{i}(x_{2})\right\Vert \\
& <C\lambda_{\varepsilon}^{j}+\varepsilon L_{g}\sum_{i=1}^{j}C\lambda
_{\varepsilon}^{i},
\end{align*}
where the last two inequalities follow from (\ref{eq:term-3-helpful-bound}).
Taking $j \to \infty$ gives%
\begin{equation}
\left\vert \pi_{I}\left( x_{\varepsilon}-z_{\varepsilon}\right) \right\vert
\leq\varepsilon\frac{\lambda_{\varepsilon}}{1-\lambda_{\varepsilon}}CL_{g}.
\label{eq:term3-bound}
\end{equation}

We now turn to the middle term from (\ref{eq:the-three-terms}). Since $%
f_{\varepsilon}\left( x\right) =f_{0}\left( x\right) +\varepsilon g\left(
x\right) $ and $\pi_{I}f_{0}\left( x\right) =x$, it follows that (below we
consider $x=f_{\varepsilon}^{j}(x_{\varepsilon})$) 
\begin{align*}
\pi_{I}\left( f_{\varepsilon}\left( f_{\varepsilon}^{j}(x_{\varepsilon
})\right) -f_{\varepsilon}^{j}(x_{\varepsilon})\right) & =\pi_{I}f_{0}\left(
f_{\varepsilon}^{j}(x_{\varepsilon})\right) +\varepsilon\pi _{I}g\left(
f_{\varepsilon}^{j}(x_{\varepsilon})\right)
-\pi_{I}f_{\varepsilon}^{j}(x_{\varepsilon}) \\
& =\varepsilon\pi_{I}g\left( f_{\varepsilon}^{j}(x_{\varepsilon})\right)\corTypo{,}
\end{align*}
so%
\begin{equation}
\pi_{I}\left( f_{\varepsilon}^{m}(x_{\varepsilon})-x_{\varepsilon}\right)
=\sum_{j=0}^{m-1}\pi_{I}\left( f_{\varepsilon}^{j+1}\left( x_{\varepsilon
}\right) -f_{\varepsilon}^{j}\left( x_{\varepsilon}\right) \right)
=\varepsilon\sum_{j=0}^{m-1}\pi_{I}g\left(
f_{\varepsilon}^{j}(x_{\varepsilon })\right) .  \label{eq:term2-bound}
\end{equation}

Combining (\ref{eq:the-three-terms}), (\ref{eq:term1-bound}), (\ref%
{eq:term3-bound}), (\ref{eq:term2-bound}) gives%
\begin{equation*}
\pi_{I}\left( f_{\varepsilon}^{m}\left( \sigma_{\alpha}^{\varepsilon}\left(
z_{\varepsilon}\right) \right) -z_{\varepsilon}\right) >\varepsilon\left(
\sum_{j=0}^{m-1}\pi_{I}g\left( f_{\varepsilon}^{j}(x_{\varepsilon})\right) -%
\frac{1+\lambda_{\varepsilon}}{1-\lambda_{\varepsilon}}CL_{g}\right) .
\end{equation*}
Since the right hand side of the inequality above depends continuously on $%
\varepsilon$, from (\ref{eq:key-assumption-in-proof}) we obtain (\ref%
{eq:shadowing-proof-key-step}) for sufficiently small $\varepsilon$.

This establishes the key step (\ref{eq:shadowing-proof-key-step}). We now
apply Theorem \ref{th:shadowing} to prove our result. Indeed, since $\omega
_{|\Lambda_{0}}$ is nondegenerate, the same is true for $\omega
_{|\Lambda_{\varepsilon}}$ for sufficiently small $\varepsilon.$ 
Choose $x_{0}\in\Lambda_{\varepsilon}$ having $%
\pi_{I}x_{0}=0$ and consider $\alpha_{0},m_{0}$, (which are allowed to
depend on $x_{0}$) such that for%
\begin{equation*}
x_{1}:=f_{\varepsilon}^{m_{0}}\circ\sigma_{\alpha_{0}}^{\varepsilon}\left(
x_{0}\right)\corTypo{,}
\end{equation*}
we have $\pi_{I}\left(
f_{\varepsilon}^{m_{0}}\circ\sigma_{\alpha_{0}}^{\varepsilon}\left(
x_{0}\right) -x_{0}\right) >c\varepsilon$. This can be done due to (\ref%
{eq:shadowing-proof-key-step}). Repeating the procedure, choosing $%
\alpha_{i},m_{i}$ for which $\left(
\pi_{I}f_{\varepsilon}^{m_{i}}\circ\sigma_{\alpha_{i}}^{\varepsilon}\left(
x_{i}\right) -x_{i}\right) >c\varepsilon$ we obtain a pseudo-orbit $%
x_{0},\ldots,x_{N}$, where $x_{i+1}:=f_{\varepsilon}^{m_{i}}\circ
\sigma_{\alpha_{i}}^{\varepsilon}\left( x_{i}\right) $, for which 
\begin{equation*}
\pi_{I}\left( x_{N}-x_{0}\right) >Nc\varepsilon.
\end{equation*}
Choosing $N$ large enough, we obtain that $\pi_{I}\left( x_{N}-x_{0}\right)
>1$. By Theorem \ref{th:shadowing} the pseudo-orbit $x_{0},\ldots,x_{N}$ is $%
\delta$-shadowed by a true orbit, so by choosing 
\begin{equation*}
\delta<\frac{1}{2}\left( \pi_{I}\left( x_{N}-x_{0}\right) -1\right) \corTypo{,}
\end{equation*}
we have the claim.
\end{proof}

In Theorem \ref{th:main} we assume that for any point in $\Lambda_{0}$ we
can find a pseudo-orbit such that we have a gain in $I$. Note however that
we do not need  to have (\ref{eq:key-assumption}) for all $z\in\Lambda_{0}$.
It is enough to have (\ref{eq:key-assumption}) for $z$ on some smaller
subset of $\Lambda_{0},$ provided that we can ensure that the pseudo-orbit
constructed in the proof of Theorem \ref{th:main} returns to that set. Below
we formulate Theorem \ref{cor:strip}, which will make this statement
precise. First we introduce one notion.

\begin{definition}
Consider the topology on $\Lambda_{0}\cap\left\{ I\in\left[ 0,1\right]
\right\} $ induced by $\Lambda_{0}$. We say that an open set $%
S\subset\Lambda_{0}\cap\left\{ I\in\left[ 0,1\right] \right\} $ is a strip
in $\Lambda_{0}$ iff%
\begin{equation*}
S\cap\left\{ z\in\Lambda_{0}:\pi_{I}z=\iota\right\} \neq\emptyset \qquad%
\text{for any }\iota\in\left[ 0,1\right] .
\end{equation*}
(Recall that we consider $\mathbb{T}=\corTypo{\mathbb{R}\, \mathrm{mod}\, 2\pi}$; the interval $I\in[0,1]$ is a strict subset of $[0,2\pi)$. Since $S$ is open in the topology induced on $\Lambda_{0}\cap\left\{ I\in%
\left[ 0,1\right] \right\} $ we require that it contains points with $I=0 $
and $I=1$.)
\end{definition}

\begin{figure}[tbp]
\begin{center}
\includegraphics[height=2.6cm]{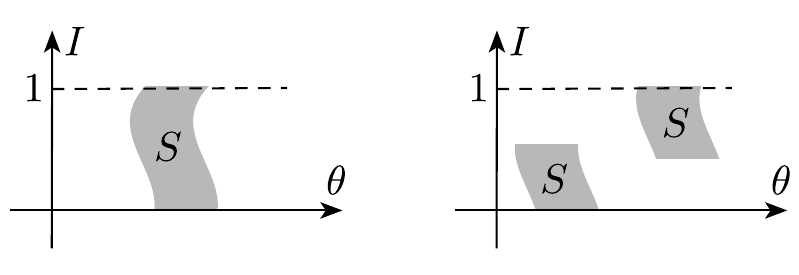}
\end{center}
\caption{A typically shaped strip (left) and a `strip' consisting of two
connected components (right). }
\label{fig:strips}
\end{figure}
We refer to $S$ as a `strip' because usually we would choose it to be of the
shape as in the left hand side of Figure \ref{fig:strips}. In principle
though a strip \corTypo{might }look differently, for instance as on the right hand side
plot in figure \ref{fig:strips}. 

In subsequent two theorems we consider two strips $S^{+}$ and $S^{-}$. The
strip $S^{+}$ is used to validate diffusion in $I$, which increases $I$ by
order one. The strip $S^{-}$ will be used to prove diffusion in which $I$
decreases by order one.

\begin{theorem}
\label{cor:strip} Assume that conditions (\ref{eq:Lip-g-assumption}) and (%
\ref{eq:contraction-expansion-bounds}) are satisfied, and that for $%
\varepsilon =0$ we have the sequence \corTypo{of }scattering maps $\sigma _{\alpha }:%
\mathrm{dom}\left( \sigma _{\alpha }\right) \rightarrow \Lambda _{0}$ for $%
\alpha =1,\ldots ,L$. 
Let $S^{+}\subset \Lambda _{0}$ be a strip\footnote{%
We add the plus in the superscript for $S^{+}$ since this strip is used to
increase $I$. In subsequent theorem we will have another strip $S^{-}$ to
obtain diffusion in the \corTypo{opposite }direction.}. Assume that for every $z\in \overline{S^{+}}$ \corTypo{:}
\begin{enumerate}
\item \corTypo{there exists a constant $m\in \mathbb{N}$ and } an $\alpha \in
\left\{ 1,\ldots ,L\right\} $ for which $z\in \mathrm{dom}\left(
\sigma _{\alpha }\right) $ \corTypo{and}
\begin{equation}
f_{0}^{m}\circ \sigma _{\alpha }\left( z\right) \in S^{+},
\label{eq:strip-return-cond}
\end{equation}%
\item  \corTypo{there exists }a point $x\in
W_{z}^{u}\left( f_{0},U\right) \cap W_{\sigma _{\alpha }\left( z\right)
}^{s}\left( f_{0}\right) $ such that $f_{0}^{m}\left( x\right) \in
W_{f_{0}^{m}\left( \sigma _{\alpha }\left( z\right) \right) }^{s}(f_{0},U)$
and%
\begin{equation}
\sum_{j=0}^{m-1}\pi _{I}g\left( f_{0}^{j}\left( x\right) \right) -\frac{%
1+\lambda }{1-\lambda }L_{g}C>0.  \label{eq:key-assumption-again}
\end{equation}
\end{enumerate}
\corTypo{(The $m,\alpha$ and $x$ can depend on $z$.) }Then for sufficiently small $\varepsilon >0$ there exists an $x_{\varepsilon
}$ and $n_{\varepsilon }>0$ such that%
\begin{equation*}
\pi _{I}\left( f_{\varepsilon }^{n_{\varepsilon }}\left( x_{\varepsilon
}\right) -x_{\varepsilon }\right) >1.
\end{equation*}
\end{theorem}

\begin{proof}
The result follows by making minor adjustments to the arguments in the proof
of Theorem \ref{th:main}. Let $S_{\varepsilon }^{+}\subset \Lambda
_{\varepsilon }$ be the perturbation of the strip $S^{+}\subset \Lambda _{0}$%
. As in the proof of Theorem \ref{th:main} we construct a pseudo orbit $%
x_{i+1}=f_{\varepsilon }^{m_{i}}\circ \sigma _{\alpha _{i}}^{\varepsilon
}\left( x_{i}\right) $, starting with a point $x_{0}\in S_{\varepsilon }^{+}$
with $\pi _{I}x_{0}=0$. Note \corTypo{that }we assume that (\ref{eq:strip-return-cond})
holds for any $z\in \overline{S^{+}}$ (with choices of $m$ and $\alpha $
depending on $z$). This means that for sufficiently small $\varepsilon $,
and for any point $z_{\varepsilon }\in S_{\varepsilon }^{+}$, there is an $%
m=m\left( z_{\varepsilon }\right) ,\alpha =\alpha \left( z_{\varepsilon
}\right) $ such that $f_{\varepsilon }^{m(z_{\varepsilon })}\circ \sigma
_{\alpha (z_{\varepsilon })}^{\varepsilon }\left( z_{\varepsilon }\right)
\in S_{\varepsilon }^{+}$. In other words, $z_{\varepsilon }$ `returns' to
the strip for sufficiently small $\varepsilon $. Due to the compactness of $%
\overline{S^{+}}$, a sufficiently small choice of $\varepsilon $ guarantees
that we have $f_{\varepsilon }^{m(z_{\varepsilon })}\circ \sigma _{\alpha
(z_{\varepsilon })}^{\varepsilon }\left( z_{\varepsilon }\right) \in
S_{\varepsilon }^{+}$ for all $z_{\varepsilon }\in S_{\varepsilon }^{+}.$ In
short, condition (\ref{eq:strip-return-cond}) ensures that the pseudo-orbit $%
x_{i+1}=f_{\varepsilon }^{m_{i}}\circ \sigma _{\alpha _{i}}^{\varepsilon
}\left( x_{i}\right) $ remains within the strip $S_{\varepsilon }^{+}$ for
sufficiently small $\varepsilon $. By (\ref{eq:key-assumption-again}) and
identical arguments to those from Theorem \ref{th:main} we therefore have 
\begin{equation*}
\pi _{I}\left( x_{i+1}-x_{i}\right) >\varepsilon c,
\end{equation*}
for some $c>0$, and the result follows from the shadowing argument\corTypo{, by applying Theorem \ref{th:shadowing}, }just as in the proof of
Theorem \ref{th:main}.
\end{proof}

A mirror result gives diffusion in the opposite direction.

\begin{theorem}
\label{cor:strip_down} Assume that conditions (\ref{eq:Lip-g-assumption})
and (\ref{eq:contraction-expansion-bounds}) are satisfied, and that for $%
\varepsilon =0$ we have the sequence of scattering maps $\sigma _{\alpha }:%
\mathrm{dom}\left( \sigma _{\alpha }\right) \rightarrow \Lambda _{0}$ for $%
\alpha =1,\ldots ,L$. Let
 $S^{-}\subset \Lambda _{0}$ be a strip. Assume that for every $z\in \overline{%
S^{-}}$\corTypo{:}
\begin{enumerate}
\item \corTypo{there exists a constant $m\in \mathbb{N}$ and }an $\alpha \in \left\{ 1,\ldots ,L\right\} $ for which $%
z\in \mathrm{dom}\left( \sigma _{\alpha }\right) $ \corTypo{and}
\begin{equation*}
f_{0}^{m}\circ \sigma _{\alpha }\left( z\right) \in S^{-},
\end{equation*}%
\item \corTypo{there exists }a point $x\in
W_{z}^{u}\left( f_{0},U\right) \cap W_{\sigma _{\alpha }\left( z\right)
}^{s}\left( f_{0}\right) $ such that $f_{0}^{m}\left( x\right) \in
W_{f_{0}^{m}\left( \sigma _{\alpha }\left( z\right) \right) }^{s}(f_{0},U)$
and%
\begin{equation*}
\sum_{j=0}^{m-1}\pi _{I}g\left( f_{0}^{j}\left( x\right) \right) +\frac{%
1+\lambda }{1-\lambda }L_{g}C<0.
\end{equation*}%
\end{enumerate}
\corTypo{(The $m,\alpha$ and $x$ can depend on $z$.) }Then for sufficiently small $\varepsilon >0$ there exists an $x_{\varepsilon
}$ and $n_{\varepsilon }>0$ such that%
\begin{equation*}
\pi _{I}\left( x_{\varepsilon }-f_{\varepsilon }^{n_{\varepsilon }}\left(
x_{\varepsilon }\right) \right) >1.
\end{equation*}
\end{theorem}

\begin{proof}
The proof follows as in the proof of Theorem \ref{cor:strip}.
\end{proof}

\corTypo{By }combining the two strips we obtain shadowing of any prescribed finite
sequence of actions.

\begin{theorem}
\label{th:shadowing-seq} Assume that two strips $S^{+}$ and $S^{-}$ satisfy
assumptions of Theorems \ref{cor:strip} and \ref{cor:strip_down},
respectively. If in addition

\begin{enumerate}
\item for every $z\in \overline{S^{+}}$ there exists an $n$ (which
can depend on $z$) such that $f_{0}^{n}(z)\in S^{-}$, and

\item for every $z\in \overline{S^{-}}$ there exists an $n$ (which
can depend on $z$) such that $f_{0}^{n}(z)\in S^{+}$,
\end{enumerate}
then \correction{comment 34}{}\cor{there exists an $M>0$, such that }for any given finite sequence $\{I_{k}\}_{k=0}^{N}$, 
and for sufficiently small $\varepsilon $ there exists an orbit of $%
f_{\varepsilon }$ which $\cor{\varepsilon M} $-shadows the actions $I_{k}$; i.e. there
exists a point $z^{\varepsilon}_{0}$ and a sequence of integers $%
n_{1}^{\varepsilon }\leq n_{2}^{\varepsilon }\leq \ldots \leq
n_{N}^{\varepsilon }$ such that%
\begin{equation*}
\left\Vert \pi _{I}f_{\varepsilon }^{n_{k}^{\varepsilon
}}(z^{\varepsilon}_{0})-I_{k}\right\Vert <\cor{\varepsilon M} .
\end{equation*}
\end{theorem}

\begin{proof}
Suppose that $I_{1}>I_{0}$. (The opposite case will be analogous.) As in the
proof of Theorem \ref{cor:strip}, we construct a pseudo orbit $%
x_{i+1}=f_{\varepsilon }^{m_{i}}\circ \sigma _{\alpha _{i}}^{\varepsilon
}\left( x_{i}\right) $, $x_{i}\in S_{\varepsilon }^{+}$, starting with a
point $x_{0}$ with $\pi _{I}x_{0}=I_{0}$, such that 
\begin{equation*}
\pi _{I}\left( x_{i+1}-x_{i}\right) >\varepsilon c,
\end{equation*}%
for some $c>0$. \cor{By smooth dependence of $f_{\varepsilon}$ and $\sigma^{\varepsilon}_{\alpha_i}$ on $\varepsilon$, and by the compactness of the normally hyperbolic manifold, we can choose large enough $M$ so that for every such $x_i$ we have 
\begin{equation}
|\pi _{I}\left( x_{i+1}-x_{i}\right)| < \varepsilon M/2. \label{eq:eps-shadowing}
\end{equation}
(The $M$ can be chosen to be independent from $x_i$.) }
We can therefore find a pseudo orbit for which $\left\vert \pi
_{I}x_{i_{1}}-I_{1}\right\vert <\cor{\varepsilon M} /2$, for some $i_1>0$. If $%
I_{2}>I_{1} $, and we carry on as in the proof of Theorem \ref{cor:strip},
continuing with our pseudo-orbit along $S_{\varepsilon }^{+}$, until we
reach $x_{i_{2}}$ such that $\left\vert \pi _{I}x_{i_{2}}-I_{2}\right\vert
<\cor{\varepsilon M} /2$. If on the other hand $I_{2}<I_{1}$, then we take $%
x_{i_{1}+1}=f_{\varepsilon }^{n}(x_{m_{l_{1}}})$, where the $n$ is the
number from assumption 1. (for $z=x_{i_{1}}$). For sufficiently small $%
\varepsilon $ we will obtain that $x_{i_{1}+1}\in S_{\varepsilon }^{-}$. We
now construct the subsequent points $x_{i}$ along the strip $S_{\varepsilon
}^{-}$, going down in $I$ along each step, until we reach $x_{i_{2}}$
satisfying $\left\vert \pi _{I}x_{i_{2}}-I_{2}\right\vert <\cor{\varepsilon M} /2$. \cor{(Possibly we might need to enlarge $M$, so that we ensure (\ref{eq:eps-shadowing}) for points from the strip $S_{\varepsilon }^{-}$.) }
Depending on whether $I_{k+1}>I_{k}$ or $I_{k+1}<I_{k}$ we \corTypo{proceed }in an
analogous manner: to move up in $I$ we construct the given fragment of the
pseudo-orbit along $S_{\varepsilon }^{+}$; and to go down in $I$ we
construct the given fragment of the pseudo-orbit along $S_{\varepsilon }^{-}$%
. Assumptions 1., 2. ensure that our pseudo-orbit can be chosen to jump
between the strips $S_{\varepsilon }^{+}$ and $S_{\varepsilon }^{-}$ at any
stage of the construction.

This way we construct a pseudo orbit for which%
\begin{equation*}
\left\vert \pi _{I}x_{i_{k}}-I_{k}\right\vert <\cor{\varepsilon M} /2\qquad \text{for }%
k=0,\ldots ,N.
\end{equation*}%
By Theorem \ref{th:shadowing} the pseudo-orbit $x_{i}$ can be $\delta$
-shadowed by a true orbit. \cor{For fixed $\varepsilon $, we can choose $\delta=\varepsilon M/2$, }which concludes our proof.
\end{proof}

\begin{remark}
\label{rem:generalisation2} \emph{Theorems \ref{cor:strip}, \ref%
{cor:strip_down}, \ref{th:shadowing-seq} can be generalised to the setting
of higher dimensional $I$ by singling out one action, as in Remark \ref%
{rem:generalisation1}. The definition of the strip is then with respect to
that particular action. }
\end{remark}



\section{Example of application \label{sec:example}}

In this section we discuss our example, the generalized standard map, to which
we apply our method. We give a computer assisted proof of the existence of
diffusing orbits by applying Theorem \ref{th:shadowing-seq}. We validate the
assumptions of the theorem using two independent implementations, which use
different methods to obtain bounds on the stable/unstable manifolds of the
NHIM. The first is based on cone conditions \cite{Z, MR2784613, MR3397322},
and the second on the parameterization method \cite{MR1976079, MR1976080,
MR2177465}.  The stable/unstable manifolds calculations are discussed in more detail in 
Section \ref{sec:cap}.

\subsection{The Generalized Standard Map}

Let $V(q)$ be a \correction{comment 35}{$(2\pi\mathbb{Z})^{n}$-periodic function}. Consider a map
$f:\mathbb{R}^{2n}\rightarrow\mathbb{R}^{2n}$ given by
\[
f(q,p)=(q+p+\nabla V(q),p+\nabla V(q)).
\]

\begin{remark}
The map $f$ is symplectic and has the generating function
\[
S(q,Q)=\frac{1}{2}\left\Vert Q-q\right\Vert ^{2}+V(q).
\]

\end{remark}

\begin{remark}
When $V=0$ the map is completely integrable. When $n=1$ and $V(q)=\alpha
\cos(q)$ then we obtain the Chirikov Standard Map.
\end{remark}

For our example, taking $q=\left(  x,\theta\right)  ,$ $p=\left(
y,I\right)  $ and%
\[
V_{\varepsilon}(x,\theta)=\alpha\cos\left(  x\right)  -\varepsilon\sin
(x)\sin(\theta),
\]
we obtain a family of maps (\ref{eq:example}).
To be in line with the setup from section \ref{sec:main-results} we interpret
that $f_{\varepsilon}:\mathbb{R}^{2}\times\mathbb{T}^{2}\rightarrow
\mathbb{R}^{2}\times\mathbb{T}^{2}$. (We could just as well interpret
$f_{\varepsilon}$ to be on $\mathbb{T}^{4}.$)

In our example we take $\alpha=4.$ For this parameter, when $\varepsilon=0$,
on the $x,y$ coordinates we have a hyperbolic fixed point at the origin. The
reader can get a sense of the dynamics by referring to the simulation results
illustrated in Figure \ref{fig:phaseSpace}.

\corTypo{For }$\varepsilon=0$ the system consists of a pair of decoupled maps
$F:\mathbb{R}^{2}\rightarrow\mathbb{R}^{2}$ and $G:\mathbb{T}^{2}%
\rightarrow\mathbb{T}^{2}$
\begin{equation}
f_{0}\left(  x,y,\theta,I\right)  =\left(  F\left(  x,y\right)  ,G\left(
\theta,I\right)  \right)  . \label{eq:F-G-f0}%
\end{equation}
The origin on the $x,y$ plane is a hyperbolic fixed point of $F$ and $DF(0)$
has eigenvalues $\lambda,\lambda^{-1}$ for $\lambda=3-2\sqrt{2}$ (here we took
$\alpha=4$).

The torus%
\[
\Lambda_{0}=\left\{  \left(  0,0,\theta,I\right)  :\theta\in\mathbb{T}%
^{1},I\in\mathbb{T}^{1}\right\}\corTypo{,}
\]
is a normally hyperbolic invariant manifold for $f_{0}$ with the rates
{$\lambda$ and 
\correction{comment 23}{$\mu=(\sqrt{5}/2+1)/2$. }(The $\mu$ is the
norm of the matrix acting on $\theta,I$ in (\ref{eq:example}) for
$\varepsilon=0$.) 

We consider the standard symplectic form%
\[
\omega=dx\wedge dy+d\theta\wedge dI.
\]
The maps $f_{\varepsilon}$ are $\omega$-symplectic and $\omega|_{\Lambda_{0}}
$ is non-degenerate.


\begin{figure}[t]
\begin{center}
\includegraphics[width=7cm]{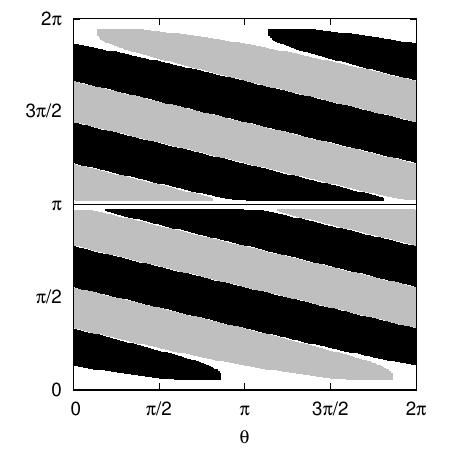}
\end{center}
\caption{The strips from Theorem \ref{th:shadowing-seq} for the map
(\ref{eq:example}), validated by our computer program. The $S^{+}$ is in black
and $S^{-}$ in \corTypo{grey}. The angle $\theta$ is on the horizontal axis and $I$ on
the vertical axis.}%
\label{fig:strips-CAP}%
\end{figure}

We prove the following result.

\begin{theorem}
[Diffusion in the generalized standard map]\label{th:main-example}
\cor{There exists an $M>0$ such that }for every finite
sequence $\left\{  I^{l}\right\}  _{l=0}^{L}\subset\left[  \frac{1}{5}%
,\pi-\frac{1}{10}\right]  $ \correction{comment 36}{and for every sufficiently small $\varepsilon>0$, there exists }a sequence of
integers $n_{1}^{\varepsilon},\ldots,n_{L}^{\varepsilon}$\corTypo{, a point $z_{0}^{\varepsilon}$, and a sequence of points $z_{l}^{\varepsilon}:=f_{\varepsilon}^{n_{l}^{\varepsilon}}\left(  z_{l-1}^{\varepsilon}\right)
$ for $l=1,\ldots,L$, such that}
\[
\left\vert \pi_{I}z_{l}^{\varepsilon}-I_{l}\right\vert <\cor{\varepsilon M},\qquad\text{for
}l=0,\ldots,L.
\]

\end{theorem}

\begin{remark}
The proof of this theorem is based on computer assisted validation of the
assumptions of Theorem \ref{th:shadowing-seq}. The strips validated by our
computer program are depicted in Figure \ref{fig:strips-CAP}.
\end{remark}

\begin{remark}
\label{rem:second-I}From our validation of the strips (see Figure
\ref{fig:strips-CAP}) it follows also that we can take the interval $\left[
\pi+\frac{1}{10},2\pi-\frac{1}{5}\right]  $ instead of $\left[  \frac{1}%
{5},\pi-\frac{1}{10}\right]  $ in Theorem \ref{th:main-example}. Between these
two intervals though, at $I=\pi$ \corTypo{and $I=0$, }we have \corTypo{gaps}, which our method is unable to
overcome. In other words, we are not able to establish an orbit which would
start with $I\in(0,\pi)$ and finish with $I\in(\pi,2\pi)$ (and vice versa). 
\end{remark}

\begin{remark}
The diffusion \corTypo{is in fact established }for intervals reaching in $I$ slightly
closer to $0$ and $\pi$ than stated in Theorem \ref{th:main-example}, where we
have rounded down the intervals. Our computer assisted proof based on the
parameterization method does a better job and produces higher (in $I$) strips
than the method based on cone conditions. This is because the parametrization
method leads to much higher accuracy of the bounds on the stable/unstable
manifolds, which is then reflected in better accuracy of the remaining
computations.
Both methods though can be used to validate the $I$-intervals stated in
Theorem \ref{th:main-example} and Remark \ref{rem:second-I}.
\end{remark}

\begin{remark}
 If we take the parameter $\alpha$ in (\ref{eq:example}) closer to zero, then the unstable eigenvalues at the origin becomes smaller and the problem becomes more challenging numerically. This is because with weak hyperbolicity it is more difficult to obtain good estimates on the manifolds; also the homoclinic excursion takes more iterates. We have found that close to $\alpha=0.15$ the method based on cone conditions fails \corTypo{at establishing the bounds for the intersection of the stable/unstable manifolds}, but the parametrization method can still be applied.
\end{remark}

\subsection{\cor{Interval Newton Method\label{sec:interval-newton}}}
\correction{comment 27}{}
In our computer assisted proofs we use the following classical result, which
allows one to conclude from the existence of a ``good enough'' approximate
solution that there exists a true solution to a nonlinear system of
equations.


\correction{comment 17}{By an interval matrix $\mathbf{A}\subset \mathbb{R}^{n\times n}$, we mean a matrix whose elements are intervals.}
Let $\mathcal{F}:\mathbb{R}^{k}\rightarrow \mathbb{R}^{k}$ be a $C^{1}$
function and $U\subset \mathbb{R}^{k}$. We shall denote by $[D\mathcal{F}%
(U)] $ the interval enclosure of a Jacobian matrix on the set $U$. This
means that $[D\mathcal{F}(U)]$ is an interval matrix defined as 
\begin{equation*}
\lbrack D\mathcal{F}(U)]=\left\{ A\in \mathbb{R}^{k\times k}|A_{ij}\in \left[
\inf_{x\in U}\frac{d\mathcal{F}_{i}}{dx_{j}}(x),\sup_{x\in U}\frac{d\mathcal{%
F}_{i}}{dx_{j}}(x)\right] \text{ for }i,j=1,\ldots ,k\text{ }\right\} .
\end{equation*}
Let $\mathbf{A}\subset\mathbb{R}^{k\times k}$ be an interval matrix. We shall write $\mathbf{A}^{-1}$ to denote an interval matrix, for which if $A\in \mathbf{A}$ then $A^{-1}\in \mathbf{A}^{-1}$.
\begin{theorem}
\label{th:interval-Newton}\cite{Alefeld} (Interval Newton method) Let $%
\mathcal{F}:\mathbb{R}^{k}\rightarrow \mathbb{R}^{k}$ be a $C^{1}$ function
and $X=\Pi _{i=1}^{k}[a_{i},b_{i}]$ with $a_{i}<b_{i}$. If $[D\mathcal{F}%
(X)] $ is invertible and there exists an $x_{0}$ in $X$ such that%
\begin{equation*}
N(x_{0},X):=x_{0}-\left[ D\mathcal{F}(X)\right] ^{-1}f(x_{0})\subset X,
\end{equation*}%
then there exists a unique point $x^{\ast }\in X$ such that $\mathcal{F}%
(x^{\ast })=0.$
\end{theorem}

\subsection{Proof of Theorem \ref{th:main-example}}

The proof of Theorem \ref{th:main-example} exploits computer assisted
validation methods for studying the local stable/unstable manifolds of fixed
points. We apply these for the map $F$ from (\ref{eq:F-G-f0}), i.e. the
unperturbed map acting on $x,y$. We take the origin as our fixed point of $F$.
The methods allow us to obtain an open interval $J\subset\mathbb{R}$ and
smooth functions $P_{u}:J\rightarrow\mathbb{R}^{2}$ and $P_{s}:J\rightarrow
\mathbb{R}^{2}$ such that $P_{u}\left(  J\right)  $ is the local unstable
manifold $W_{0}^{u}\left(  F,U\right)  $ of the origin for $F$, and
$P_{s}\left(  J\right)  $ is the local stable manifold $W_{0}^{s}\left(
F,U\right)  $ of the origin for $F$, for some neighbourhood $U$ of the origin.
We give a description of both methods in section \ref{sec:cap}. For the
purpose of this section it is enough that we can obtain explicit bounds for
such functions, as well as for their first derivatives. Moreover, the methods
allow us to obtain explicit bounds $C,\lambda\in\mathbb{R}$, $C,\lambda>0$
such that%
\begin{equation}%
\begin{array}
[c]{r}%
\left\Vert F^{i}\left(  P_{s}(x)\right)  \right\Vert \leq C\lambda
^{i}\smallskip\\
\left\Vert F^{-i}\left(  P_{u}(x)\right)  \right\Vert \leq C\lambda^{i}%
\end{array}
\qquad\text{for all }i\in\mathbb{N}\text{ and }x\in J.
\label{eq:C-bound-example}%
\end{equation}

The functions $P_{u}$ and $P_{s}$ give only a local description of the
unstable and stable manifolds. To establish their intersections we use the
following parallel shooting approach.  Define $\mathcal{F}:\overline
{J}\times B_{1}\times\ldots\times B_{M}\times\overline{J}\rightarrow
\mathbb{R}^{2M+2}$, where $B_{i}\subset\mathbb{R}^{2}$ are cartesian products
of two closed intervals, as
\begin{align*}
&  \mathcal{F}\left(  x,v_{0},\ldots,v_{M-1},y\right) \\
&  :=(P_{u}(x)-v_{0},F\left(  v_{0}\right)  -v_{1},\ldots,F\left(
v_{M-2}\right)  -v_{M-1},F\left(  \corTypo{v_{M-1}}\right)  -P_{s}(y)).
\end{align*}
If we establish the existence of a point $p^{\ast}=\left(  x^{\ast}%
,v_{0}^{\ast},\ldots,v_{M-1}^{\ast},y^{\ast}\right)  $ for which
\begin{equation}
\mathcal{F}\left(  p^{\ast}\right)  =0, \label{eq:F-zero-homoclinic}%
\end{equation}
then we have established a sequence of points $v_{0}^{\ast},\ldots,v_{M}%
^{\ast}$, where $v_{1}^{\ast}=P_{u}(x^{\ast})$ and $v_{M}^{\ast}=P_{s}(y)$,
along a homoclinic \corTypo{orbit }to zero. The bound on the solution of
(\ref{eq:F-zero-homoclinic}) can be established by using the interval Newton
theorem\footnote{An alternative could be to use the Newton-Krawczyk theorem or a version of the Newton-Kantorovich theorem. We use the interval Newton theorem because of its simplicity and the fact that it is sufficient for our needs in this particular example.}; see section \ref{sec:interval-newton}. This way, we obtain a
homoclinic orbit within a set of the form%
\begin{equation}
v_{i}^{\ast}\in\left[  x_{i}^{0}-r,x_{i}^{0}+r\right]  \times\left[  y_{i}%
^{0}-r,y_{i}^{0}+r\right]  \qquad\text{for }i=0,\ldots,M\corTypo{,}
\label{eq:bounds-points}%
\end{equation}
where $x_{i}^{0},y_{i}^{0}$ are written in Table \ref{table:homoclinic}. (Our
$M$ is equal to $10$.)\begin{table}[t]
\caption{Homoclinic orbit}
\centering
\begin{tabular}{l l l}
\hline\hline
$i$ & $x_i^0$ & $y_i^0$ \\ [0.5ex] 
\hline
0 & 0.003855589164542 & 0.003194074612644  \\
1 & 0.022471982225036 & 0.018616393060494  \\
2 & 0.130968738959384 & 0.108496756734347 \\
3 & 0.761844080808229 & 0.630875341848845 \\
4 & 4.153747139236954 & 3.391903058428725  \\
5 & 4.153747139236954 & 0.000000000000001 \\
6 & 0.761844080808229 & -3.391903058428725 \\
7 & 0.130968738959384 & -0.630875341848845 \\
8 & 0.022471982225036 & -0.108496756734347 \\
9 & 0.003855589164542 & -0.018616393060494 \\
10 & 0.000661514551898 & -0.003194074612644  \\
[1ex]
\hline
\end{tabular}
\label{table:homoclinic}
\end{table}

We use two methods to obtain bounds on $P_{u}$ and $P_{s}$. In the case of the
first method, by using cones, we obtain%
\begin{equation}
r=r_{\text{cones}}=1.5\cdot10^{-7}, \label{eq:r-cone}%
\end{equation}
and by using the second method, i.e. the parameterization method, we obtain
\begin{equation}
r=r_{\text{param}}=6.5\cdot10^{-15}. \label{eq:r-param}%
\end{equation}
(The bounds on our computer program are in fact often tighter and vary \correction{comment 38}{from
point to point}. Here we have rounded them up to write a uniform enclosure $r$
for all considered points.)

Since we use the interval Newton method as the tool for our validation we also
obtain transversality of obtained intersection of our manifolds. (Such results
are well known, see for instance \cite{MR3068557} for a similar approach. We
add the proof in the appendix to keep the work self-contained.)

\begin{lemma}
\label{lem:manifolds-intersect-standard}The manifolds $W_{0}^{u}\left(
F\right)  $ and $W_{0}^{s}\left(  F\right)  $ intersect transversally.
\end{lemma}

\begin{proof}
The proof is given in \ref{app:trans}.
\end{proof}

Define the sequence
\[
\left(  x_{i}^{\ast},y_{i}^{\ast}\right)  :=F^{i}\left(  v_{0}^{\ast}\right)
\qquad\text{for all }i\in\mathbb{Z}.
\]
Note that $\left(  x_{i}^{\ast},y_{i}^{\ast}\right)  =v_{i}^{\ast}$, for
$i=0,\ldots,M$. We now show that for $\varepsilon=0$ \corTypo{the map }(\ref{eq:example})
has a well defined homoclinic channel with a global scattering map.

\begin{lemma}
\label{lem:scattering-map-defined}The set
\[
\Gamma=\left\{  \left(  x_{0}^{\ast},y_{0}^{\ast},I,\theta\right)
:I,\theta\in\mathbb{T}^{1}\right\}  ,
\]
is a homoclinic channel for $f_{0}$ and the associated scattering map $\sigma$
is globally defined and is the identity on $\Lambda_{0}$.
\end{lemma}

\begin{proof}
To show that $\Gamma$ is a homoclinic channel for $f_{0}$ we need to prove
points (i), (ii) and (iii) from Definition \ref{def:homoclinic-channel}.

We start by observing that for $p\in\Gamma$
\begin{equation}
T_{p}\Gamma=\left\{  \left(  0,0\right)  \right\}  \times\mathbb{R}^{2}.
\label{eq:tangent-Gamma}%
\end{equation}
Since $W_{0}^{u}\left(  F\right)  ,$ $W_{0}^{s}\left(  F\right)  $ intersect
transversally in $\mathbb{R}^{2}$ at $v_{0}^{\ast}$ we also have
\begin{align}
T_{v_{0}^{\ast}}W_{0}^{s}\left(  F\right)  \oplus T_{v_{0}^{\ast}}W_{0}%
^{u}\left(  F\right)   &  =\mathbb{R}^{2},\label{eq:Transversality-F}\\
T_{v_{0}^{\ast}}W_{0}^{s}\left(  F\right)  \cap T_{v_{0}^{\ast}}W_{0}%
^{u}\left(  F\right)   &  =\left\{  0\right\}  . \label{eq:Transversality-F-2}%
\end{align}
Since $W_{\Lambda}^{u}\left(  f_{0}\right)  =W_{0}^{u}\left(  F\right)
\times\mathbb{T}^{2}$ and $W_{\Lambda}^{s}\left(  f_{0}\right)  =W_{0}%
^{s}\left(  F\right)  \times\mathbb{T}^{2}$ we see that for $p\in\Gamma$
\begin{align}
T_{p}W_{\Lambda}^{u}\left(  f_{0}\right)   &  =T_{v_{0}^{\ast}}W_{0}%
^{u}\left(  F\right)  \times\mathbb{R}^{2},\label{eq:Tangent-Wu-f0}\\
T_{p}W_{\Lambda}^{s}\left(  f_{0}\right)   &  =T_{v_{0}^{\ast}}W_{0}%
^{s}\left(  F\right)  \times\mathbb{R}^{2}. \label{eq:Tangent-Ws-f0}%
\end{align}
From (\ref{eq:Transversality-F}), (\ref{eq:Tangent-Wu-f0}),
(\ref{eq:Tangent-Ws-f0}) and (\ref{eq:Transversality-F-2}),
(\ref{eq:Tangent-Wu-f0}), (\ref{eq:Tangent-Ws-f0}), (\ref{eq:tangent-Gamma})
we obtain, respectively,
\begin{align*}
T_{p}W_{\Lambda}^{s}\left(  f_{0}\right)  \corTypo{+} T_{p}W_{\Lambda}^{u}\left(
f_{0}\right)   &  =\mathbb{R}^{4},\\
T_{p}W_{\Lambda}^{s}\left(  f_{0}\right)  \cap T_{p}W_{\Lambda}^{u}\left(
f_{0}\right)   &  =\left\{  (0,0)\right\}  \times\mathbb{R}^{2}=T_{p}\Gamma,
\end{align*}
which proves (i) from Definition \ref{def:homoclinic-channel}.

Since any two points that converge to each other need to start with the same
values on $\theta,I$ we see that for any $z\in\Lambda_{0}$%
\begin{equation}
W_{z}^{u}\left(  f_{0}\right)    =W_{0}^{u}\left(  F\right)  \times\left\{
\pi_{\left(  \theta,I\right)  }z\right\}  ,\label{eq:fiber-full-1}
\end{equation}
\corTypo{and} 
\begin{equation}
W_{z}^{s}\left(  f_{0}\right)    =W_{0}^{s}\left(  F\right)  \times\left\{
\pi_{\left(  \theta,I\right)  }z\right\}  . \label{eq:fiber-full-2}%
\end{equation}
This means that the wave maps are of the form
\begin{equation}
\Omega_{\pm}\left(  \corTypo{x,y},\theta,I\right)  =\left(
0,0,\theta,I\right)  . \label{eq:wave-maps-example}%
\end{equation}
Clearly $\corTypo{(\Omega_{\pm})|_{\Gamma}}$ are diffeomorphisms as required in (iii) from
Definition \ref{def:homoclinic-channel}.

From (\ref{eq:fiber-full-1}), (\ref{eq:fiber-full-2}) we see that for any
$p\in\Gamma$ and $z\in\Lambda_{0}$%
\begin{equation}
T_{p}W_{z}^{u}\left(  f_{0}\right)     =T_{v_{0}^{\ast}}W_{0}^{u}\left(
F\right)  \times\left\{  \left(  0,0\right)  \right\}
,\label{eq:tangent-at-fiber-1}
\end{equation}
\corTypo{and}
\begin{equation}
T_{p}W_{z}^{s}\left(  f_{0}\right)     =T_{v_{0}^{\ast}}W_{0}^{s}\left(
F\right)  \times\left\{  \left(  0,0\right)  \right\} \corTypo{.}
\label{eq:tangent-at-fiber-2}%
\end{equation}
Combining (\ref{eq:tangent-Gamma}) with (\ref{eq:tangent-at-fiber-1}),
(\ref{eq:tangent-at-fiber-2}) and comparing with (\ref{eq:Tangent-Wu-f0}),
(\ref{eq:Tangent-Ws-f0}) gives%
\begin{equation}
T_{p}\Gamma\oplus T_{p}W_{z}^{u}\left(  f_{0}\right)     =T_{v_{0}^{\ast}%
}W_{0}^{u}\left(  F\right)  \times\mathbb{R}^{2}=T_{p}W_{\Lambda}^{u}\left(
F\right)  ,
\end{equation}
\corTypo{and}
\begin{equation}
T_{p}\Gamma\oplus T_{p}W_{z}^{s}\left(  f_{0}\right)     =T_{v_{0}^{\ast}%
}W_{0}^{s}\left(  F\right)  \times\mathbb{R}^{2}=T_{p}W_{\Lambda}^{s}\left(
F\right),
\end{equation}
which means that we have (ii) from Definition \ref{def:homoclinic-channel}. We
have established that $\Gamma$ is a homoclinic chanel. From
(\ref{eq:wave-maps-example}) we see that the associated scattering map
$\sigma$ is globally defined and is the identity on $\Lambda_{0}$.
\end{proof}

We validate \corTypo{the }strips $S^{+}$ and $S^{-}$ with \corTypo{the }shapes as in Figure
\ref{fig:strips-CAP}. These are composed of small overlapping rectangular
fragments. Below we introduce a lemma which we then apply on each such
rectangular part. First we introduce a notation. For $a,b\in\lbrack0,2\pi)$ we
define the interval $\left[  a,b\right]  \subset\mathbb{T}=\corTypo{\mathbb{R}\,\operatorname{mod}2\pi}$ as%
\begin{equation}
\left[  a,b\right]  =\left\{
\begin{array}
[c]{lll}%
\left\{  x\in\mathbb{T}:a\leq x\leq b\right\}  &  & \text{if }a\leq b,\\
\left\{  x\in\mathbb{R}:b\leq x\leq a+2\pi\right\}  \text{ mod }2\pi &  &
\text{if }b<a.
\end{array}
\right.  \label{eq:interval-on-torus}%
\end{equation}
We define $\left(  a,b\right)  \subset\mathbb{T}\mathbb{\ }$as the
interior of $\left[  a,b\right]  $.

Let $I_{1},I_{2}\in(0,2\pi)$ satisfy $I_{1}<I_{2}$. Let $s_{1},s_{2}%
\in\mathbb{T}$, and consider strips on $\Lambda_{0}$ of the form%
\begin{equation}
\{(0,0)\} \times\left[  s_{1},s_{2}\right]  \times\left[  I_{1},I_{2}\right]
. \label{eq:strip-between-s1-s2}%
\end{equation}
(In (\ref{eq:strip-between-s1-s2}) the interval $\left[  s_{1},s_{2}\right]  $
is in the sense (\ref{eq:interval-on-torus}).) We now have the following lemma.


\begin{lemma}
\label{lem:strip-reformulation-example}
If
\begin{equation}
\sum_{i=0}^{M-1}\sin(x_{i}^{\ast})\cos(\theta+iI)>3\frac{1+\lambda}{1-\lambda
}C \corTypo{,}\label{eq:key-assumption-for-example}
\end{equation}
and if for every $\left(  \theta,I\right)  \in\left[  s_{1},s_{2}\right]
\times\left[  I_{1},I_{2}\right]  $ there exists an $m\geq M$ (the $m$ can
depend on the choice of $\left(  \theta,I\right)  $) such that%
\begin{equation}
\theta+mI\in\left(  s_{1},s_{2}\right)  ,\label{eq:point-returns-to-strip}%
\end{equation}
then assumptions of Theorem \ref{cor:strip} hold true for our map
(\ref{eq:example}) on the strip (\ref{eq:strip-between-s1-s2}).
\end{lemma}

\begin{proof}
Condition (\ref{eq:strip-return-cond}) follows from
(\ref{eq:point-returns-to-strip}). We need to validate
(\ref{eq:key-assumption-again}).
Since $v_{M}^{\ast}\in P_{s}(J),$ from (\ref{eq:C-bound-example}) it follows
that $|x_{m}^{\ast}|<C\lambda^{m-M},$ for $m\geq M$. 

Consider an arbitrary fixed $\left(  \theta,I\right)  \in\left[  s_{1}%
,s_{2}\right]  \times\left[  I_{1},I_{2}\right]  $ and let
\[
C_{m}:=\sum_{j=0}^{m-1}\sin(x_{j}^{\ast})\cos(\theta+jI).
\]
Since for $j\geq M$ we know that $|x_{j}^{\ast}|<C\lambda^{j-M},$ we see that
for $m\geq M$
\begin{equation}
|C_{m}-C_{M}|\leq\sum_{j=M}^{m-1}\left\vert \sin(x_{j}^{\ast})\right\vert
\left\vert \cos(\theta+jI)\right\vert \leq C\frac{1-\lambda^{m-M}}{1-\lambda
}<C\frac{1+\lambda}{1-\lambda}.\label{eq: sum_tail_bound}%
\end{equation}

Observe that the map $\left(  x,y,\theta,I\right)  \rightarrow\sin
(x)\cos\left(  \theta\right)  $ is Lipschitz with the constant $L_{g}=2$.

For $z=\left(  0,0,\theta,I\right)  \in\{(0,0)\}\times\left[  s_{1}%
,s_{2}\right]  \times\left[  I_{1},I_{2}\right] $, consider $x=\left(
x_{0}^{\ast},y_{0}^{\ast},\theta,I\right)  \in W_{z}^{u}\left(  f_{0}%
,U\right)  \cap W_{\sigma_{\alpha}\left(  z\right)  }^{s}\left(  f_{0}\right)
$. Since $\left(  x_{0}^{\ast},y_{0}^{\ast}\right)  =v_{0}^{\ast}\in P_{u}(J)$
and $v_{M}^{\ast}\in P_{s}(J)$, for every $m\geq N$, $f_{0}^{m}\left(
x\right)  \in W_{f_{0}^{m}\left(  \sigma_{\alpha}\left(  z\right)  \right)
}^{s}(f_{0},U)$. Also, for every $m\geq M$, by using
(\ref{eq:key-assumption-for-example}) and (\ref{eq: sum_tail_bound}), we
obtain
\begin{align*}
\sum_{j=0}^{m-1}\pi_{I}g\left(  f_{0}^{j}\left(  x\right)  \right)
-\frac{1+\lambda}{1-\lambda}L_{g}C &  =\sum_{j=0}^{m-1}\sin(x_{j}^{\ast}%
)\cos(\theta+jI)-2\frac{1+\lambda}{1-\lambda}C\\
&  \geq C_{M}-\left\vert C_{m}-C_{M}\right\vert -2\frac{1+\lambda}{1-\lambda
}C\\
&  \geq C_{M}-3\frac{1+\lambda}{1-\lambda}C\\
&  >0,
\end{align*}
which ensures
(\ref{eq:key-assumption-again}). This finishes our proof.
\end{proof}

\begin{remark}
\label{rem:mirror-S-}
A mirror result lets us validate assumptions
of Theorem \ref{cor:strip_down}. The only difference is that instead of
(\ref{eq:key-assumption-for-example}), we require
\[
\sum_{i=0}^{N-1}\sin(x_{i}^{\ast})\cos(\theta+iI)<-3\frac{1+\lambda}%
{1-\lambda}C.
\]
\end{remark}

We are now ready to prove Theorem \ref{th:main-example}.

\begin{proof}
[Proof of Theorem \ref{th:main-example}]By Lemma
\ref{lem:manifolds-intersect-standard} the stable and unstable manifolds of
the origin for the map $F$ intersect transversally. Moreover, we have explicit
bounds for a homoclinic orbit along this intersection, written in Table
\ref{table:homoclinic} and (\ref{eq:bounds-points}--\ref{eq:r-param}). This
means that, by Lemma \ref{lem:scattering-map-defined}, the scattering map for
the unperturbed system is well defined.

Using the bounds from Table 1 and (\ref{eq:bounds-points}--\ref{eq:r-param}),
which give an enclosure of a finite fragment of the homoclinic orbit, and
together with the aid of Lemma \ref{lem:strip-reformulation-example}, our
computer program constructs the strip $S^{+}$ from Figure \ref{fig:strips-CAP}%
. This strip is a union of overlapping rectangles, for which assumptions of
Theorem \ref{cor:strip} are satisfied. We use a mirror result to Lemma
\ref{lem:strip-reformulation-example} (see Remark \ref{rem:mirror-S-}), to
construct the strip $S^{-}$ from Figure \ref{fig:strips-CAP}, for which
\corTypo{assumptions }of Theorem \ref{cor:strip_down} are satisfied. We also validate
that for these two strips conditions 1. and 2. of Theorem
\ref{th:shadowing-seq} are \corTypo{fulfilled}.

After such validation the result follows from Theorem \ref{th:shadowing-seq}.
\end{proof}

The computer assisted proof using cone conditions for the validation of intersections of the manifolds was performed with the CAPD\footnote{Computer Assisted Proofs in Dynamics: http://capd.ii.uj.edu.pl} library \cite{capd}. The parameterization method approach was implemented in Matlab. The source code is available on the web page of the corresponding author.

\section*{Acknowledgements}We would like to thank the anonymous Reviewers for their comments, suggestions and corrections, which helped us improve our paper.
\appendix


\section{Modification of a system with a normally hyperbolic invariant
cylinder to one with a normally hyperbolic invariant torus\label{sec:gluing}}

\begin{figure}[ptb]
\begin{center}
\includegraphics[height=4cm]{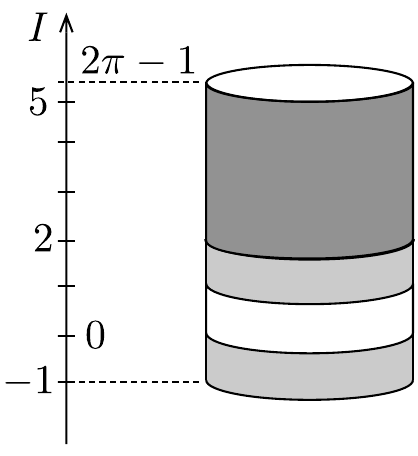}
\end{center}
\caption{For $I\in[0,1]$ the system is not modified (white area). In the light grey regions
the system is modified by the `bump' function. \cor{For $I\in[2,2\pi -1]$ we `freeze' the system to be the unperturbed map, which is represented by the dark grey area.}}%
\label{fig:gluing}%
\end{figure}

\mymarginpar{\hyperlink{Return_comment 18}{comment 18}}
\cor{
Consider a family of maps $f_{\varepsilon}:\mathbb{R}^{2d}%
\times\mathbb{R\times T}^{1}\rightarrow\mathbb{R}^{2d}\times\mathbb{R\times
T}^{1}$ to be a parameter dependent family of the $2\pi$ time shift along the trajectory maps (or section-to-section maps), as described in section \ref{sec:main}, for a
Hamiltonian system
\[
x^{\prime}=J\nabla_x\left(  H\left(  x\right)  +\varepsilon G(x,t)\right),
\]
where $H:\mathbb{R}%
^{2d+2} \rightarrow \mathbb{R}$ and $G:\mathbb{R}%
^{2d+2} \times \mathbb{T}\rightarrow \mathbb{R}$.
We consider $I(x)=H(x)$ as the preserved quantity for $\varepsilon=0$, and our
coordinates are $x=\left(  u,s,I,\theta\right)  \in\mathbb{R}^{2d}%
\times\mathbb{R\times T}^{1}$. We assume that for $f_{\varepsilon=0}$ the manifold $\Lambda_0=\{0\}\times\R\times\T$ is an invariant cylinder. 

We will modify the system so that we obtain a
map $\tilde{f}_{\varepsilon}$ defined on $\mathbb{R}^{2d}\times\mathbb{T}^{2}$
for which
\[
f_{\varepsilon}\left(  u,s,I,\theta\right)  =\tilde{f}_{\varepsilon}\left(
u,s,I,\theta\right)  \qquad\text{for }I\in\left[  0,1\right]  .
\]

We start by explaining the idea,
which is depicted in Figure \ref{fig:gluing}. For $I\in\left[  0,1\right]  $
we leave the system as it is. We then employ a `bump' function so that at the
edges of the domain $I\in\left[  -1,2\right]  $, i.e. for $I=2$ and
$I=-1=2\pi-1$, we have $\tilde{f}_{\varepsilon}=f_{0}$. For the remaining
$I\in (2,2\pi-1)$ we `freeze' the system taking $\tilde{f}_{\varepsilon}=f_{0}$.

In detail, we consider a smooth `bump' function\footnote{For
instance $b(x)=\exp(-\left(  1-x^{2}\right)  ^{-1})$ for $x\in\left[
-1,0\right]  $, $b(x)=1$ for $x\in\left[  0,1\right]  $, $b(x)=\exp(-\left(
1-( 1-x)^{2} \right)  ^{-1})$ for $x\in\left[  1,2\right]  $ and zero
otherwise.} $b:\mathbb{R}\rightarrow\left[  0,1\right]  $ for which%
\begin{align*}
b\left(  I\right)   &  =0\qquad\text{for }I\in\mathbb{R}\setminus\left(
-1,2\right), \\
b\left(  I\right)   &  =1\qquad\text{for }I\in\left[  0,1\right],
\end{align*}
and take $\tilde f_{\varepsilon}$ to be the $2\pi$ time shift maps (or section-to-section maps) for a modified ODE
\[
x^{\prime}=J\nabla_x\left(  H\left(  x\right)  +b(H(x)) \varepsilon G(x,t)\right).
\]
Such modification can allow
us to apply our results directly in the below considered cases:

\begin{case} If on $\Lambda_{\varepsilon}$ we obtain for $\tilde f_{\varepsilon}|_{\Lambda_{\varepsilon}}$ a Cantor set of KAM tori, then the region on $\Lambda_{\varepsilon}$ between every two invariant tori  constitutes an invariant set for $\tilde f_{\varepsilon}$. This means that we can obtain diffusing orbits by means of Theorems \ref{th:main}, \ref{cor:strip}, \ref{cor:strip_down} and \ref{th:shadowing-seq}  for $\tilde f_{\varepsilon}$ which are in $\{I\in (0,1)\}$. (This is because orbits in $\Lambda_{\varepsilon}$ resulting from Poincar\'e recurrence, which are used for the shadowing construction from Theorem \ref{th:shadowing}, will be contained between KAM tori, which we can choose to be in $\{I\in (0,1)\}$.) Since in $\{I\in (0,1)\}$ the maps $\tilde f_{\varepsilon}$ and $f_{\varepsilon}$ coincide, we obtain diffusing orbits for $f_{\varepsilon}$.
\end{case}

\begin{case}
If we are interested in proving that for sufficiently small $\varepsilon>0$
there exist orbits that change in $I$ by more than $\frac{1}{3}$, then we can
use the following dichotomy:

\begin{enumerate}
\item There exists an orbit in $\Lambda_{\varepsilon}$ which changes in $I$ by more than $\frac{1}{3}$ for the map $f_{\varepsilon}|_{\Lambda_{\varepsilon}}$; then there is nothing to prove.

\item There are no orbits in $\Lambda_{\varepsilon}$, which diffuse in $I$ by more than $\frac{1}{3}$ for the map $f_{\varepsilon}|_{\Lambda_{\varepsilon}}$. Then orbits of $f_{\varepsilon}$ starting from
$\Lambda_{\varepsilon}\cap \left \{I\in\left[  \frac{1}{3},\frac{2}{3}\right]\right \}$ will not leave $\Lambda_{\varepsilon}\cap\left\{ I\in\left( 0,1\right)\right\}$. On $\left\{ I\in\left( 0,1\right)\right\}$ we know that $f_{\varepsilon}=\tilde f_{\varepsilon}$. We
can use Theorem \ref{th:main} or Theorem \ref{cor:strip} to construct orbits that change in $I$ by more than $\frac{1}{3}$ for $\tilde f_{\varepsilon}$. The constructed orbits will remain in $\{I\in (0,1)\} $. (This is because orbits in $\Lambda_{\varepsilon}$ for Poincar\'e recurrence used for the construction from Theorem \ref{th:shadowing} will be contained in $\{I\in (0,1)\}$.) On $\{I\in (0,1)\}$ we know that $\tilde f_{\varepsilon}=  f_{\varepsilon}$, so any orbit changing in $I$ by $\frac{1}{3}$ for the map $\tilde f_{\varepsilon}$ does so also for $f_{\varepsilon}$.

\end{enumerate}
\end{case}

\begin{remark}
The second case is easily generalised to higher dimensions. We either 1) diffuse along some action on $\Lambda_{\varepsilon}$, or 2) diffuse in the selected action which is controlled by Theorem \ref{th:main} or Theorem \ref{cor:strip}.

The first case is not easily generalised, since we can have gaps between KAM tori in higher dimensions. 
\end{remark}
}



\section{Invariant manifolds and their intersections\label{sec:manifolds}}
\label{sec:cap} 

We now discuss computation of the
 local stable and \corTypo{local }unstable manifolds, with a focus on obtaining
 mathematically rigorous 
computer assisted error bounds on all approximations.\correction{comment 39}{}

\cor{We focus on two methods, which we used for two independent validations of the intersection of the manifolds. The first is based on cone cone conditions, and the second on the parameterisation method. The first is simpler, and we provide full details, simplifying the results to our particular setting. For the parameterisation method we restrict to Remark \ref{rem:parmMethod} where we point our reader to the relevant references.}



Below described method of cones is based on the more general results from \cite{Z}.
We reformulate these for our particular setting, simplifying and giving sketches of
proofs, in order to keep the paper self-contained.

Let $F$ be the map (\ref{eq:F-G-f0}), i.e. the unperturbed map $f_{0}$ acting
on $x,y$\corTypo{; where we recall that $\alpha = 4$. }Let $\mathcal{P}\in\mathbb{R}^{2\times2}$ and $\tilde{F}:\mathbb{R}^{2}\rightarrow\mathbb{R}^{2}$ 
 be defined as follows
\[
\mathcal{P}:=\left(
\begin{array}
[c]{cc}%
1+\sqrt{2} & 1-\sqrt{2}\\
2 & 2
\end{array}
\right)  ,\qquad\tilde{F}\left(  z\right)  :=\mathcal{P}^{-1}F\left(
\mathcal{P}z\right).
\]
The matrix $\mathcal{P}$ is the coordinate change to Jordan form for $DF(0)$
and $\tilde{F}(z)$ is the map expressed in local coordinates, which
\corTypo{diagonalize }the stable and unstable directions at the origin, i.e.
$D\tilde{F}\left(  0\right)  =\mathrm{diag}\left(  (3-2\sqrt{2})^{-1}%
,3-2\sqrt{2}\right)$.  We refer to these as the local coordinates, \corTypo{and write }
$z=\left(  u,s\right)  $. (The $u$ stands for `unstable' and
$s$ for `stable'.)

Let $\mathcal{L\in}\mathbb{R}$ be a fixed constant satisfying $\mathcal{L}>0$
and define $C:\mathbb{R}^{2}\rightarrow\corTypo{\mathbb{R}}$ as
\[
C\left(  u,s\right)  =\mathcal{L}\left\vert u\right\vert -\left\vert
s\right\vert .
\]
For $z\in\mathbb{R}^{2}$ we define the cone at $z$ as $C^{+}\left(  z\right)
:=\left\{  v:C\left(  z-v\right)  \geq0\right\}  $ (see Figure \ref{fig:cone}%
). Let $r>0$ be fixed, $J:=\left[  -r,r\right]  \subset\mathbb{R}$ and let
$B\subset\mathbb{R}^{2}$ be the rectangle $B:=\left[  -r,r\right]
\times\left[  -\mathcal{L}r,\mathcal{L}r\right]  .$

\begin{figure}[ptb]
\begin{center}
\includegraphics[width=8.5cm]{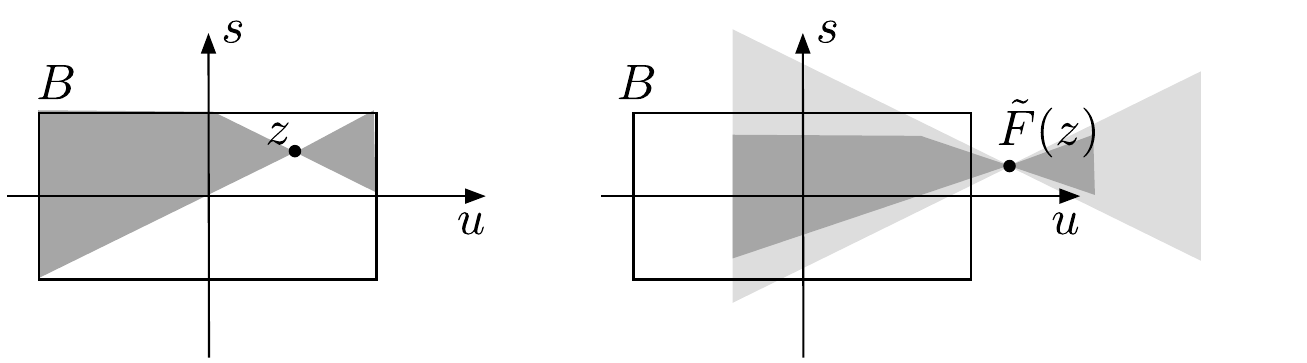}
\end{center}
\caption{The cone at $z$ intersected with $B$ (in dark grey) is mapped into
the cone at $\tilde F (z)$ (in light grey).}%
\label{fig:cone}%
\end{figure}

\begin{definition}
We say that $\tilde{F}$ satisfies cone conditions in $B$ if for every $z\in B$
we have (see Figure \ref{fig:cone})%
\[
\tilde{F}\left(  C^{+}(z)\cap B \right)  \subset C^{+}\left(  \tilde{F}\left(
z\right)  \right)  .
\]

\end{definition}

We have the following lemma, which gives bounds on the unstable manifold in the local coordinates.

\begin{lemma}
\label{lem:cone-cond}If $\tilde{F}$ satisfies cone conditions in $B$, and
there exists a $\lambda<1$ such that for every $z\in C^{+}\left(  0\right)  $
we have
\begin{equation}
|\pi_{u}\tilde{F}(u,s)|>\lambda^{-1}\left\vert u\right\vert
,\label{eq:expansion-in-cone}%
\end{equation}
then there exists a smooth function $w:J\rightarrow\left[  -r\mathcal{L}%
,r\mathcal{L}\right]  $, such that
\[
W_{0}^{u}(\tilde{F},B)=\left\{  \left(  u,w(u)\right)  :u\in J\right\}.
\]
Moreover, $\left\vert \frac{d}{du}w(u)\right\vert \leq\mathcal{L}$ and for
every $u\in J$%
\begin{equation}
\left\Vert \tilde{F}^{-n}\left(  u,w(u)\right)  \right\Vert <\lambda^{n}%
\sqrt{1+\mathcal{L}^{2}}\left\vert u\right\vert .\label{eq:back-contr-bound}%
\end{equation}

\end{lemma}

\begin{proof}
Since $0$ is a hyperbolic fixed point of $\tilde{F}$, locally at the fixed
point the unstable manifold exists, is smooth, and tangent to the horizontal
axis, hence it is contained in $C^{+}\left( 0\right) $. Cone condition
together with (\ref{eq:expansion-in-cone}) ensure that the unstable manifold
is \corTypo{stretched }through $B$ to become a graph above $J$. Since locally, close to
zero, the unstable manifold is tangent to the horizontal axis it is a graph
of a function with the Lipschitz constant smaller than $\mathcal{L}$. This
property is preserved as the manifold is stretched throughout $B$ thanks to
the cone condition.

To show (\ref{eq:back-contr-bound}) note that for $z\in C^{+}\left( 0\right) 
$, since $\left\vert \pi_{s}z\right\vert <\mathcal{L}\left\vert \pi
_{u}z\right\vert $, we obtain $\left\Vert z\right\Vert \leq\sqrt{1+\mathcal{L}%
^{2}}\left\vert \pi_{u}z\right\vert $. Thus, from (\ref{eq:expansion-in-cone}%
),%
\begin{equation*}
\left\Vert z\right\Vert <\sqrt{1+\mathcal{L}^{2}}\left\vert
\pi_{u}z\right\vert <\sqrt{1+\mathcal{L}^{2}}\lambda\left\vert \pi_{u}%
\tilde {F}(z)\right\vert . 
\end{equation*}
Taking $z=\tilde{F}^{-n}\left(\corTypo{ u, w\left( u\right)} \right) $ and using (\ref%
{eq:expansion-in-cone}) we obtain%
\begin{equation*}
\left\Vert \tilde{F}^{-n}\left(\corTypo{u, w\left( u\right)} \right) \right\Vert <\sqrt{%
1+\mathcal{L}^{2}}\lambda\left\vert \pi_{u}\tilde{F}^{-n+1}(\corTypo{u,w\left( u\right)}
)\right\vert <\ldots<\sqrt{1+\mathcal{L}^{2}}\lambda^{n}\left\vert
u\right\vert , 
\end{equation*}
as required.
\end{proof}

In practice we can validate cone conditions and (\ref{eq:expansion-in-cone})
from the interval enclosure of the derivative of $\tilde{F}$ on $B$. 

\begin{lemma}
\label{lem:cones-check-1}If $[D\tilde{F}\left(  B\right)  ]\left(
C^{+}\left(  0\right)  \right)  \subset C^{+}\left(  0\right)  $ then
$\tilde{F}$ satisfies cone conditions.
\end{lemma}
\begin{proof} Let $z\in B$ and $v\in C^{+}(z)\cap B$. Since $v-z\in C^{+}(0)$, from our
assumption it follows that 
\begin{align}
\tilde{F}(v)-\tilde{F}(z) & =\int_{0}^{1}\frac{d}{dt}\tilde{F}\left(
z+t\left( v-z\right) \right) dt  \label{eq:DF-integral} \\
& =\int_{0}^{1}D\tilde{F}\left( z+t\left( v-z\right) \right) dt\left(
v-z\right) \in\left[ DF\left( B\right) \right] \left( v-z\right) \subset
C^{+}(0),  \notag
\end{align}
hence $\tilde{F}\left( C^{+}(z)\right) \subset C^{+}(\tilde{F}(z))$, as
required.
\end{proof}

Above lemma is \corTypo{straightforward }to apply in interval arithmetic by checking that
\[[D\tilde{F}\left(  B\right)  ]\left(  \{1\}\times\left[  -\mathcal{L}%
,\mathcal{L}\right]  \right)  \subset C^{+}\left(  0\right)  .\]

\begin{lemma}
\label{lem:cones-check-2}Let $a_{11},a_{12},a_{21},a_{22}$ be real intervals
such that $[D\tilde{F}\left(  B\right)  ]=(a_{ij})_{i,j\in\left\{
1,2\right\}  }$. If $a_{11}-\mathcal{L}\left\vert a_{12}\right\vert
>\lambda^{-1}$ then (\ref{eq:expansion-in-cone}) is \corTypo{fulfilled}.
\end{lemma}
\begin{proof} Let $\left( u,s\right) \in C^{+}(0)\cap B$. From a mirror argument to (\ref%
{eq:DF-integral}) and since $\left\vert s\right\vert \leq\mathcal{L}%
\left\vert u\right\vert ,$ 
\begin{equation*}
\left\vert \pi_{u}F\left( u,s\right) \right\vert \in\left\vert \pi _{u}\left[
DF\left( B\right) \right] \left( u,s\right) \right\vert \geq
a_{11}\left\vert u\right\vert -\mathcal{L}\left\vert a_{12}\right\vert
\left\vert u\right\vert >\lambda^{-1}\left\vert u\right\vert , 
\end{equation*}
as required.
\end{proof}

Using a computer program we compute an interval enclosure
 $[D\tilde{F}(B)]$. This enclosure is
used to validate, via Lemmas \ref{lem:cones-check-1} and
\ref{lem:cones-check-2}, the assumptions of Lemma \ref{lem:cone-cond}. This
way we obtain $w:J\rightarrow [-r\mathcal{L},r\mathcal{L}]$, and
define $P_u:J\to \mathbb{R}^2$ by
\[
P_{u}(x):=\mathcal{P}\left(  x,w(x)\right)  .
\]
Note that since $w(x)$ is Lipschitz with constant $\mathcal{L}$, $w(x)\in [-\mathcal{L}%
x,\mathcal{L}x]$, our method allows us to
obtain the explicit bound%
\[
P_{u}(x)\subset\mathcal{P}\left(  \{x\}\times\left[  -\mathcal{L}%
x,\mathcal{L}x\right]  \right), \qquad \mbox{for every } x\in J.
\]
Moreover, by Lemma \ref{lem:cone-cond} we know that $\frac{d}{dx}\left(
x,w(x)\right)  \in\{1\}\times\left[  -\mathcal{L},\mathcal{L}\right]  $, which
gives the bound on the derivative of $P_{u}$ as
\[
\frac{d}{dx}P_{u}(x)\subset\mathcal{P}\left(  \{1\}\times\left[
-\mathcal{L},\mathcal{L}\right]  \right) , \qquad \mbox{for every } x\in J .
\]
From (\ref{eq:back-contr-bound}) we also see that for every $x\in J$
\[
\left\Vert F^{-n}\left(  P_{u}(x)\right)  \right\Vert =\left\Vert
\mathcal{P}\tilde{F}^{-n}\left(  x,w(x)\right)  \right\Vert \leq\left\Vert
\mathcal{P}\right\Vert \lambda^{n}\sqrt{1+\mathcal{L}^{2}}\left\vert
x\right\vert \leq C\lambda^{n},
\]
for $C:=\left\Vert \mathcal{P}\right\Vert \sqrt{1+\mathcal{L}^{2}}r$; recall
that $J=\left[  -r,r\right]  $. We thus see that we have all the bounds for
$P_{u}$, which are required by section \ref{th:main-example}.

The function $P_{s}$ and associated bounds can be obtained the same way, by
considering $F^{-1}$ instead of $F$.

\begin{remark}[Parameterization method for invariant manifolds] \label{rem:parmMethod}\corTypo{
As $\alpha$ approaches zero in the standard map, the fixed point becomes
very weakly hyperbolic (eigenvalues approach unity) and hence the dynamics
near the fixed point becomes very slow.  As a result, homoclinic excursions take 
more time, and become difficult to track using only the linear approximation 
of the stable/unstable manifolds.   

To overcome these difficulties we can use -- for smaller $\alpha$ -- the approach 
developed in \cite{MR3068557}, a work which is itself based on the parameterization method of
\cite{Cabre1,Cabre2,MR2177465}.  Using the parameterization method
we compute high order Taylor expansions of the stable/unstable manifolds,
along with validated error bounds on all truncation errors. 
This strategy leads to a mathematically rigorous representation of the 
stable/unstable manifold  which is valid in a large 
neighbourhood of the fixed point.
We refer the interested reader to the book of \cite{Cana}
for much more complete discussion of the parameterization method.}
\end{remark}




\section{Proof of Lemma \protect\ref{lem:manifolds-intersect-standard}\label{app:trans}}

We will show that the tangent lines to $W_{0}^{u}\left( F\right) $ and $%
W_{0}^{s}\left( F\right) $ at the intersection point $v_{M}^{\ast}$ span $%
\mathbb{R}^{2}.$ Note that $v_{M}^{\ast}=F^{M}\left( P_{u}\left( x^{\ast
}\right) \right) $. \corTypo{Defining inductively the sequence of vectors $w_0,\ldots,w_{M-1}\in\mathbb{R}^{2}$ as }$w_{0}:=DP_{u}\left( x^{\ast}\right) 
$ and $w_{k}:=DF\left( v_{k}^{\ast}\right) w_{k-1}$, we
see that 
\begin{equation*}
\frac{d}{dx}F^{M}\left( P_{u}\left( x\right) \right) |_{x=x^{\ast}}=w_{M-1}. 
\end{equation*}
If $w_{M-1}$ was \corTypo{collinear }with $\frac{d}{dy}P_{s}\left( y\right)
|_{y=y^{\ast}}$, then there would exist an $\alpha\neq0$ for which $\frac {d%
}{dy}P_{s}\left( y^{\ast}\right) =\alpha w_{M-1}$. Taking the vector $%
V=\left( 1,w_{0},\ldots,w_{M-1},1/\alpha\right) $ would lead to 
\begin{equation*}
D\mathcal{F}\left( p^{\ast}\right) V=0. 
\end{equation*}
This is a contradiction, since if $p^{\ast}$ is validated by the use of
Theorem \ref{th:interval-Newton}, \corTypo{then }the matrix $D\mathcal{F}\left( p^{\ast
}\right) $ must be invertible.

\bibliographystyle{elsarticle-num} 
\bibliography{papers}





\end{document}